\documentclass[12pt,a4paper,oldfontcommands]{amsart}
\usepackage[utf8]{inputenc}
\usepackage[foot]{amsaddr}
\usepackage[T1]{fontenc}
\usepackage{microtype}
\usepackage[dvips]{graphicx}
\usepackage{xcolor}
\usepackage{times}
\usepackage{amsmath,amsfonts,amssymb,amsthm}
\usepackage{bm}
\usepackage{xr}
\usepackage[
breaklinks=true,colorlinks=true,
linkcolor=black,urlcolor=black,citecolor=black,% PRINT
bookmarks=true,bookmarksopenlevel=2]{hyperref}

\usepackage{geometry}
\newcommand{\ovl}{\overline{l}}
\newcommand{\ovlt}{\overline{l}_\tau}
\newcommand{\ovlth}{\overline{l}_{\tau h}}
\newcommand{\ovlsigma}{\overline{l}_\sigma}
\newcommand{\ovls}{\overline{l}_\sigma}

\newcommand{\ovv}{\overline{\varphi}}
\newcommand{\ovvt}{\overline{\varphi}_\tau}
\newcommand{\ovvth}{\overline{\varphi}_{\tau h}}
\newcommand{\ovvs}{\overline{\varphi}_\sigma}
\newcommand{\ovd}{\overline{d}}
\newcommand{\ovdt}{\overline{d}_\tau}
\newcommand{\ovdth}{\overline{d}_{\tau h}}
\newcommand{\ovds}{\overline{d}_\sigma}
\newcommand{\ovz}{\overline{z}}
\newcommand{\ovzt}{\overline{z}_\tau}
\newcommand{\ovzth}{\overline{z}_{\tau h}}
\newcommand{\ovzs}{\overline{z}_\sigma}
\newcommand{\ovp}{\overline{p}}
\newcommand{\ovpt}{\overline{p}_\tau}
\newcommand{\ovpth}{\overline{p}_{\tau h}}
\newcommand{\ovps}{\overline{p}_\sigma}

\newcommand{\Hzwei}{H^2(\Omega)}

\newcommand{\Heinsnull}{H^1_0(\Omega)}
\newcommand{\Lzwo}{L^2(\Omega)}

\newcommand{\Lzweizwei}{L^2(0,T;\Lzwo)}
\newcommand{\Leinszwei}{L^1(0,T;\Lzwo)}

\newcommand{\Heinszwei}{H^1(\I;\Lzwo)}

\newcommand{\I}{0,T}
\newcommand{\Xtau}{X^0_\tau}
\newcommand{\Vtau}{V^0_\tau}
\newcommand{\B}{\bm{B}}
\newcommand{\vt}{\varphi_\tau}
\newcommand{\dt}{d_\tau}
\newcommand{\evt}{e_\tau^\varphi}
\newcommand{\edt}{e_\tau^d}
\newcommand{\etavt}{\eta_\tau^\varphi}
\newcommand{\xivt}{\xi_\tau^\varphi}
\newcommand{\etadt}{\eta_\tau^d}
\newcommand{\etazt}{\eta_\tau^z}
\newcommand{\etapt}{\eta_\tau^p}
\newcommand{\xidt}{\xi_\tau^d}
\newcommand{\zt}{z_\tau}
\newcommand{\pt}{p_\tau}
\newcommand{\Vh}{V_h^1}
\newcommand{\Xh}{X^1_h}
\newcommand{\Vth}{V^{0,1}_{\tau h}}
\newcommand{\Xth}{X^{0,1}_{\tau h}}
\newcommand{\vth}{\varphi_{\tau h}}
\newcommand{\dth}{d_{\tau h}}
\newcommand{\evh}{e_h^\varphi}
\newcommand{\edh}{e_h^d}
\newcommand{\pth}{p_{\tau h}}
\newcommand{\zth}{z_{\tau h}}
\newcommand{\etavh}{\eta_h^\varphi}
\newcommand{\xivh}{\xi_h^\varphi}
\newcommand{\etadh}{\eta_h^d}
\newcommand{\xidh}{\xi_h^d}
\newcommand{\xiph}{\xi^p_h}
\newcommand{\xizh}{\xi^z_h}
\newcommand{\etazh}{\eta^z_h}
\newcommand{\etaph}{\eta^p_h}
\newcommand{\ezt}{e_\tau^z}
\newcommand{\ept}{e_\tau^p}
\newcommand{\ezth}{e_{\tau h}^z}
\newcommand{\epth}{e_{\tau h}^p}

\newtheorem{theorem}{Theorem}[section]

\newtheorem{proposition}[theorem]{Proposition}
\newtheorem{corollary}[theorem]{Corollary}
\newtheorem{lemma}[theorem]{Lemma}

\newtheorem{remark}[theorem]{Remark}

\newtheorem{assumption}[theorem]{Assumption}

\begin{document}
\title[Optimal control of a linear PDE-ODE system]{A priori error estimates for the space-time finite element discretization of an optimal control problem governed by a coupled linear PDE-ODE system}
\author{Marita Holtmannspötter$^1$}
\author{Arnd Rösch$^1$}
\address[$^1$]{Faculty of Mathematics, University of Duisburg-Essen, Essen, Germany}
\email{marita.holtmannspoetter@uni-due.de, arnd.roesch@uni-due.de}
\email{vexler@ma.tum.de}
\author{Boris Vexler$^2$}
\address[$^2$]{Department of Mathematics, Technical University of Munich, Garching, Germany} 
\maketitle
\begin{abstract}
In this paper we investigate a priori error estimates for the space-time Galerkin finite element discretization of an optimal control problem governed by a simplified linear gradient enhanced damage model. The model equations are of a special structure as the state equation consists of an elliptic PDE which has to be fulfilled at almost all times coupled with an ODE that has to hold true in almost all points in space. The state equation is discretized by a piecewise constant discontinuous Galerkin method in time and usual conforming linear finite elements in space. For the discretization of the control we employ the same discretization technique which turns out to be equivalent to a variational discretization approach. We provide error estimates of optimal order both for the discretization of the state equation as well as for the optimal control. Numerical experiments are added to illustrate the proven rates of convergence. 
\end{abstract}

\section{Introduction}
In this paper, we derive a priori error estimates for the space-time finite element discretization of a simplified linear gradient enhanced damage model and the associated optimal control problem. To be more specific, we investigate the finite element approximation of the optimal control problem
\[J(\varphi,d,l)=\frac{1}{2}\Vert \varphi-\varphi_d\Vert^2_{\Lzweizwei}+\frac{1}{2}\Vert d-d_d\Vert^2_{\Lzweizwei}+\frac{\alpha_l}{2}\Vert l\Vert^2_{\Lzweizwei}\]
subject to the state equation
\begin{align}
\label{mod:linmodbegin} -\alpha \Delta \varphi (t)+\beta \varphi(t)&=\beta d(t)+l(t)  \quad \text{ in } \Omega\\ 
\varphi(t)&=0 \quad \text{ on } \partial\Omega \\
 \partial_t d(t)&=-\frac{\beta}{\delta}(d(t)-\varphi (t))  \quad \text{ a. e. in } \Omega \label{mod: ode}\\
d(0)&=d_0 \label{mod:linmodend}
\end{align}
for almost all $t\in I=[\I]$ where $l$ acts as a control and $\varphi$ and $d$ are the resulting states. A precise formulation is given in the later sections. For the discretization of the state equation we will use a discontinuous piecewise constant finite element method in time and usual $H^1$-conforming linear finite elements in space. The state equation is motivated by a specific gradient enhanced damage model, first developed in \cite{DH08,DH11} and thoroughly analyzed from a mathematical point of view in \cite{MS16i,MS16ii}. First of all, this model describes the displacement of a body $\Omega$ influenced by a given force $l$. In addition, the model features two damage variables $\varphi$ and $d$ where the first one is more regular in space whereas the second one carries the evolution of damage in time. Both are coupled by a penalty term in the free energy functional with $\beta$ being the penalty parameter. The parameter $\alpha$ originates from the gradient enhancement while $\delta$ is a viscosity parameter (see \cite{MS16i} for details). The resulting system consists of two nonlinear PDEs which have to hold true in almost all time points and an ODE that should be fulfilled in almost every point in space. All three equations are fully coupled with each other. For a first analysis of the discretization of such a model we simplified the underlying PDE system, skipping the displacement variable $u$ as well as the nonlinear material function. In this first contribution we will study the linear equation \eqref{mod: ode} instead of a nonsmooth equation. This linear model problem is interesting in its own way as the equations still have the special structure of the original damage model which differs from other coupled PDE-ODE systems examined in related work, see below. The aim of this paper is to establish a priori  discretization error estimates. The more difficult nonsmooth equation will be subject of later work.  \\
The optimal control problem is formulated with a tracking type functional. To derive necessary optimality conditions, we investigate the control-to-state operator $S\colon l\mapsto (\varphi,d)$. For the discretization we employ a variational discretization technique based on \cite{MH05}. \\
Let us have a look at related work: There are quite a few contributions available regarding the optimal control of coupled PDE-ODE systems, cf. \cite{KGH18,BKR17,MST16,KG16,KG15,CPetal2009} and the references therein. The authors mainly focus on the analysis of their specific model and the derivation of first order necessary optimality conditions and provide tailored algorithms for the numerical solution of the optimal control problems. They do not derive discretization error estimates. In \cite{HV03}, the authors deal with the optimal control of laser surface hardening of steel and provide error estimates for a POD Galerkin approximation of the state equation. Error estimates for the optimal control of a coupled PDE-ODE system describing the velocity tracking problem for the evolutionary Navier–Stokes equations are derived in \cite{CC12,CC16} as well as companion papers. Here, the authors require a coupling of the discretization parameters in time and space for the well-posedness of their discretization technique. We emphasize, that in our contribution the discretization parameters can be chosen independently of one another. Our discretization setting is closely related to the techniques analyzed in \cite{MV08,NV11,MV18} for the space-time discretization of linear and semilinear parabolic optimal control problems, respectively. In these contributions the optimal control problem is not constrained by a coupled PDE-ODE system but rather by a single parabolic PDE. Therefore only one variable which carries the evolution in both space and time is considered. Error estimates for uncontrolled parabolic equations are given in \cite{EJT85,EJ91,EJ95}. \\
  
The paper is organized as follows: In section two we state the precise setting of the state equation and present the chosen discretization strategy. In section three we focus on the numerical analysis of the state equation and prove convergence of first order with respect to time and convergence of second order with respect to space of our discretization technique. We will have a look at the associated optimal control problem in section four and provide error estimates for the optimal control. The last section presents numerical examples for both the state equation and the optimal control problem. \\ 
\section{Mathematical preliminaries}
In this section we establish the principal assumptions on the data, the used notation and the chosen discretization strategy. \\

%\begin{assumption}
Throughout this paper, let $\Omega\subset \mathbb{R}^N,N\in\{2,3\}$, be a convex polygonal domain with boundary $\partial \Omega$ and let $T>0$ be a given real number. The time interval will be denoted by $I:=(\I)$. Moreover, let $\alpha,\beta,\delta>0$ be given parameters. The initial state $d_0$ is, unless otherwise stated, a function in $\Lzwo$. The right-hand side $l$ should belong to $\Lzweizwei$. The first state $\varphi$, also referred to as nonlocal state (see \cite{DH08}), is an element of the state space $V:=L^2(\I;\Heinsnull)$. The second state $d$, also called local state, should belong to $X:=H^1(\I;\Lzwo)$. \\

%\end{assumption}
We use the following short notation for inner products and norms on $\Lzwo$ and $\Lzweizwei$:
\begin{alignat*}{2}
(v,w)&:=(v,w)_{\Lzwo}, \qquad (v,w)_{I\times\Omega} &&:=(v,w)_{\Lzweizwei}, \\
\Vert v\Vert &:=\Vert v\Vert_{\Lzwo}, \qquad \qquad \Vert v\Vert_{I\times\Omega} &&:= \Vert v\Vert_{\Lzweizwei}. \\
\end{alignat*}
The norm on $L^\infty(\I;\Lzwo)$ will be abbreviated by $\Vert v\Vert_{\infty,2}$. \\
We will work with the weak formulation of the problem which reads as follows: For a given right-hand side $l\in\Lzweizwei$ and initial state $d_0\in\Lzwo$ find states $(\varphi,d)\in V\times X$ satisfying
\begin{equation} \label{eq:weakformcont}
B((\varphi,d),(\psi,\lambda))=(l,\psi)_{I\times\Omega} \qquad \forall (\psi,\lambda)\in V\times X
\end{equation}
and the initial condition $d(0)=d_0$ where the bilinear form $B$ is given as
\begin{equation} \label{eq:bilinformcont}
B((\varphi, d),(\psi,\lambda))=\alpha (\nabla \varphi,\nabla\psi)_{I\times\Omega}-\beta(d-\varphi,\psi)_{I\times\Omega}+(\partial_t d,\lambda)_{I\times\Omega}+\frac{\beta}{\delta}(d-\varphi,\lambda)_{I\times\Omega}.
\end{equation}

For the discretization in time we will employ discontinuous piecewise constant finite elements. Therefore, we consider a partition of the time interval $\overline{I}=[\I]$ as
\[\overline{I}=\{0\}\cup I_1\cup ... \cup I_M\]
with subintervals $I_m=(t_{m-1},t_m]$ of length $\tau_m$ and time points
\[0=t_0<t_1<...<t_{M-1}<t_M=T.\]
We set $\tau:=\max \{\tau_m : m=1,...,M\}$. The semidiscrete trial and test spaces are given as
\[\Vtau:=\{v_\tau\in V: v_{\tau\vert_{I_m}}\in\mathbb{P}_0(I_m;\Heinsnull), m=1,...,M\},\]
\[\Xtau:=\{d_\tau\in \Lzweizwei : d_{\tau\vert_{I_m}}\in\mathbb{P}_0(I_m;\Lzwo), m=1,...,M\}.\]
Note, that $\Vtau\subset V$ but $\Xtau\not\subset X$. Moreover, $\Vtau$ is dense in $\Xtau$ due to the dense embedding of $\Heinsnull\overset{\text{d}}{\hookrightarrow}\Lzwo$. We use the notation
\[(v,w)_{I_m\times\Omega}:=(v,w)_{L^2(I_m;\Lzwo)} \qquad \text{ and } \qquad \Vert v\Vert_{I_m\times\Omega} := \Vert v\Vert_{L^2(I_m;\Lzwo)}.\]
To express the jumps possibly occurring at the nodes $t_m$ we define
\[v^+_{\tau,m}:=\lim\limits_{t\rightarrow 0^+} v_{\tau}(t_m+t), \quad v^-_{\tau,m}:=\lim\limits_{t\rightarrow 0^+} v_{\tau}(t_m-t)=v_\tau(t_m), \quad [v_\tau]_m=v^+_{\tau,m}-v^-_{\tau,m}.\]
Note, that for functions piecewise constant in time the definition reduces to
\[v^+_{\tau,m}=v_\tau(t_{m+1})=:v_{\tau,m+1}, \qquad v^-_{\tau,m}=v_{\tau}(t_m)=: v_{\tau,m},\quad [v_\tau]_m=v_{\tau,m+1}-v_{\tau,m}.\]
The semidiscrete bilinear form $\B:(\Vtau\times\Xtau)^2\to\mathbb{R}$ is given as
\begin{align*}
\B((\varphi_\tau,d_\tau),(\psi,\lambda))&=\alpha(\nabla \varphi_\tau,\nabla \psi)_{I\times\Omega}-\beta (d_\tau-\varphi_\tau,\psi)_{I\times\Omega} \\
&+\sum\limits_{m=1}^M(\partial_t d_\tau,\lambda)_{I_m\times\Omega}+\frac{\beta}{\delta}(d_\tau-\varphi_\tau,\lambda)_{I\times\Omega}+\sum\limits_{m=2}^M([d_\tau]_{m-1},\lambda_{m-1}^+)+(d_{\tau,0}^+,\lambda^+_0).
\end{align*} 
Then, the semidiscrete state equation reads as follows: Find states $(\vt,\dt)\in\Vtau\times\Xtau$ such that
\begin{equation} \label{eq:semidiscretestatelinear}
\B((\varphi_\tau,d_\tau),(\psi,\lambda))=(l,\psi)_{I\times\Omega}+(d_0,\lambda^+_0)
\end{equation}
holds true for all $(\psi,\lambda)\in \Vtau\times\Xtau$. \\

We will require the interpolation/projection onto $\Xtau$ and $\Vtau$, respectively.  Therefore, we define the semidiscrete interpolation operator $\pi_\tau\colon C(\overline{I};\Lzwo)\to\Xtau$ with $\pi_\tau d_{\vert_ {I_m}}\in \mathbb{P}_0(I_m;\Lzwo)$ via $(\pi_\tau d)(t_m)=d(t_m)$ for $m=1,...,M$. For the projection we employ the standard $L^2$-projection in time $P_\tau \colon \Lzweizwei\to\Xtau$ given by $P_\tau\varphi_{\vert_{I_m}}:=\frac{1}{\tau_m}\int\limits_{I_m} \varphi(t)dt$. Both operators will always be denoted by the same symbols despite possibly different domains and ranges. Note, that if $\varphi\in V$ then $P_\tau\varphi\in\Vtau$ as integration in time preserves the spatial regularity due to the definition of the Bochner integral. In particular, we have \begin{equation} \label{eq:nablaPtvarphi}(\varphi-P_\tau\varphi,\psi)_{I\times\Omega}=(\nabla \varphi-\nabla P_\tau\varphi,\nabla\psi)_{I\times\Omega}=0\end{equation} for any $\psi\in\Vtau$.\\

Next, we introduce the spatial discretization. We use $H^1$-conforming finite elements in space. Thus, we consider a quasi-uniform mesh $\mathbb{T}_h$ of shape regular triangles $\mathcal{T}$, which do not overlap and cover the domain $\Omega$. By $h_\mathcal{T}$ we denote the size of the triangle $\mathcal{T}$ and $h$ is the maximal triangle size. On the mesh $\mathbb{T}_h$ we construct two conforming finite element spaces 
\[\Vh=\{v\in C(\overline{\Omega}) : v_{\vert \mathcal{T}}\in \mathbb{P}_1(\mathcal{T}), \mathcal{T}\in\mathbb{T}_h, v_{\vert\partial\Omega}=0\},\]
\[\Xh=\{v\in C(\overline{\Omega}) : v_{\vert \mathcal{T}}\in \mathbb{P}_1(\mathcal{T}), \mathcal{T}\in\mathbb{T}_h\}.\]
Then, the space-time discrete finite element spaces are given by
\[\Vth=\{v\in L^2(\I;\Vh) : v_{\vert I_m}\in\mathbb{P}_0(I_m;\Vh)\}\subset \Vtau,\]
\[\Xth=\{v\in L^2(\I;\Xh) : v_{\vert I_m}\in\mathbb{P}_0(I_m;\Xh)\}\subset \Xtau.\]
The time-space discretized state equation then reads as follows: Find states $(\vth,\dth)\in\Vth\times\Xth$ such that
\begin{equation} \label{eq:fullydiscretestate}
\B((\vth,\dth),(\psi,\lambda))=(l,\psi)_{I\times\Omega}+(d_0,\lambda^+_0)
\end{equation}
holds true for all $(\psi,\lambda)\in\Vth\times\Xth$. Note, that although we set $X=\Heinszwei$ as the state space for $d$ we choose a piecewise linear and continuous approximation in space for $d$. This is due to the fact, that we will show higher spatial regularity of $d$ in the next section, that is $d(t)\in\Hzwei$ for all $t\in[\I]$, provided that the initial datum $d_0$ possesses this regularity. \\

For the projection onto $\Vth$ and $\Xth$ we work with the standard $L^2$-projections $P^V_h\colon\Lzwo\to\Vh,P^X_h\colon\Lzwo\to\Xh$ in space on each subinterval $I_m$ and define the time-space projections $\pi^V_h\colon\Vtau\to\Vth,\pi^X_h\colon\Xtau\to\Xth$ via $(\pi^V_h z)(t)=P^V_h(z(t))$ and $(\pi^X_h z)(t)=P^X_h(z(t))$ respectively. Furthermore, based on the Ritz-projection $R_h\colon\Heinsnull\to \Vh$ given as usual by
\[(\nabla R_h\varphi,\nabla\psi)=(\nabla\varphi,\nabla\psi) \qquad \forall \psi\in\Vh\]
we define an alternative time-space projection onto $\Vth$ via $\rho_h\colon\Vtau\to\Vth, (\rho_h z)(t)=R_h(z(t))$.
\section{Numerical Analysis of the state equation}
\subsection{The continuous problem}
We first address the unique solvability of \eqref{eq:weakformcont}.
\begin{proposition} \label{prop:solvweakformcont}
For a fixed right-hand side $l\in\Lzweizwei$ and initial state $d_0\in \Lzwo$ there exists a unique solution $(\varphi,d)\in V\times X$ of equation \eqref{eq:weakformcont}. The solution exhibits the improved regularity
\begin{align*}
\varphi \in &L^2(\I; H^2(\Omega)\cap H^1_0(\Omega)) \\
d \in &\Heinszwei \hookrightarrow C(\overline{I};\Lzwo).
\end{align*}
Furthermore, if $d_0\in\Hzwei$, then we have $d\in H^1(\I;\Hzwei)$.
\end{proposition}
\begin{proof}
The proof is similar to the proof of a corresponding result for the original damage model, see \cite{MS16i,S17}. We only sketch the essential steps as we will need the notation later. If we skip the dependence of time in \eqref{mod:linmodbegin} it is well known (cf. \cite{GV85,LE98}), that there exists a unique solution $\varphi=\Phi(l,d)\in\Hzwei\cap\Heinsnull$ for every $d,l\in\Lzwo$ and 
\begin{equation} \label{eq:abschvarphicont}
\Vert\varphi\Vert_{\Hzwei}\leq C\{\Vert d\Vert_{\Lzwo}+\Vert l\Vert_{\Lzwo}\}
\end{equation} 
holds true with a constant $C>0$ only depending on the problem data but not on $d$ or $l$. Here, $\Phi\colon \Lzwo\times\Lzwo\to\Heinsnull,\Phi\colon (l,d)\mapsto \varphi,$ denotes the solution operator of the elliptic PDE
\begin{equation} \label{eq:weakvarphicont}
\alpha(\nabla\varphi,\nabla \psi)+\beta(\varphi,\psi)=(\beta d+l,\psi) \quad \forall \psi\in \Heinsnull.
\end{equation} Thus, $\varphi\colon[\I]\to \Heinsnull$, $\varphi(t):=\Phi(l(t),d(t)) \text{ f.a.a. } t\in[\I],$ satisfies \eqref{mod:linmodbegin} for every fixed $l,d\in\Lzweizwei$.  For Bochner-measurable $l$ and $d$ it is Bochner-measurable as well since $\Phi$ is Lipschitz-continuous. The regularity $\varphi\in L^2(\I;\Hzwei\cap\Heinsnull)$ may be concluded for $l,d\in\Lzweizwei$ , if we square and integrate in time on both sides of \eqref{eq:abschvarphicont}.\\
Next, we reduce the ODE onto the variable $d$ via
\begin{equation} \label{eq:reducedODEcontinuous}
\partial_td(t)=-\frac{\beta}{\delta}(d(t)-\Phi(l(t),d(t))), \qquad d(0)=d_0.
\end{equation}
The  reduced right-hand side $f\colon[\I]\times\Lzwo\to\Lzwo, f(t,d)=-\frac{\beta}{\delta}(d-\Phi(l(t),d))$ is well-defined and Lipschitz continuous with respect to the second argument for almost all $t\in [\I]$. Moreover, the Nemytskii-operator associated to $f$ maps $\Lzweizwei$ to $\Lzweizwei$. Therefore, the application of Picard-Lindelöf's theorem in abstract function spaces (see \cite{EE04} or \cite{S17}, Lem 5.7) yields the existence of a unique solution $d\in X$. The continuity of $d$ in time is a standard result (see for example \cite{WM11}, Thm. 3.1.41) for the Bochner space $H^1(\I;\Lzwo)$. If we have an initial datum $d_0\in\Hzwei$ then Picard-Lindelöf's theorem also yields $d\in H^1(\I;\Hzwei)$.
\end{proof}
\begin{assumption} \label{ass:regularityvarphi} For the rest of this section, we assume that $\varphi$ possesses higher regularity in time, that is we assume $\varphi\in\Heinszwei$.  
\end{assumption}
We will comment on this assumption in the next section.
\subsection{Semidiscretization in time}
%Recall that the semidiscrete bilinear form $\B:(\Vtau\times\Xtau)^2\to\mathbb{R}$ is given as
%\begin{align*}
%\B((\varphi_\tau,d_\tau),(\psi,\lambda))&=\alpha(\nabla \varphi_\tau,\nabla \psi)_{I\times\Omega}-\beta (d_\tau-\varphi_\tau,\psi)_{I\times\Omega} \\
%&+\sum\limits_{m=1}^M(\partial_t d_\tau,\lambda)_{I_m\times\Omega}+\frac{\beta}{\delta}(d_\tau-\varphi_\tau,\lambda)_{I\times\Omega}+\sum\limits_{m=2}^M([d_\tau]_{m-1},\lambda_{m-1}^+)+(d_{\tau,0}^+,\lambda^+_0)
%\end{align*} 
%and that the semidiscrete state equation reads as follows: Find states $(\vt,\dt)\in\Vtau\times\Xtau$ such that
%\begin{equation} 
%\B((\varphi_\tau,d_\tau),(\psi,\lambda))=(l,\psi)_{I\times\Omega}+(d_0,\lambda^+_0)
%\end{equation}
%holds true for all $(\psi,\lambda)\in \Vtau\times\Xtau$. \\
We will prove the existence of a unique solution of the semidiscrete state equation \eqref{eq:semidiscretestatelinear} as well as stability estimates for a slightly more general problem by adding an additional right-hand side $(f,\lambda)_{I\times\Omega}$ for a function $f\in\Lzweizwei$:
\begin{equation}
\label{eq:semidiscretestatelineargeneralized}
\B((\varphi_\tau,d_\tau),(\psi,\lambda))=(l,\psi)_{I\times\Omega}+(f,\lambda)_{I\times\Omega}+(d_0,\lambda^+_0) \qquad \forall (\psi,\lambda)\in \Vtau\times\Xtau.
\end{equation}
\begin{proposition} \label{prop: semidiscretestateexistence}
Let $l,f\in\Lzweizwei$ and $d_0\in\Lzwo$ be given. Then \eqref{eq:semidiscretestatelineargeneralized} possesses a unique solution $(\vt,\dt)\in\Vtau\times\Xtau$. Moreover, if $d_0\in\Hzwei$ and $f\in L^2(\I;\Hzwei)$, then we also have $\dt\in L^2(\I;\Hzwei)$. 
\end{proposition}
\begin{proof}
The solution $\vt\in\Vtau$ is given as $\vt(t)=\sum\limits_{m=1}^M \varphi_{\tau,m}\chi_{I_m}(t)$ with $\varphi_{\tau,m}:=\Phi(P_\tau l_{\vert_{I_m}},d_{\tau,m})$ while the existence of $d_{\tau,m}\in\Lzwo$ can be proven by applying Banach's fixed point theorem to the reduced fixed point equation in $\Lzwo$
\begin{equation} \label{eq:fixedpointeqoneintervallinear}
d_{\tau,m}=d_{\tau,m-1}-\frac{\beta}{\delta}\tau_m(d_{\tau,m}-\Phi(P_\tau l_{\vert_{I_m}},d_{\tau,m}))+\int\limits_{I_m} f(t)dt
\end{equation}
on each subinterval $I_m,m=1...,M,$ starting with $d_{\tau,0}=d_0$. We add $\frac{\beta}{\delta}\tau_m d_{\tau,m}$ on both sides to arrive at 
\begin{equation} \label{eq:fixedpointeqoneintervallinear}
d_{\tau,m}=F_m(d_{\tau,m}): = \frac{1}{1+\frac{\beta}{\delta}\tau_m} \left( d_{\tau,m-1}+\frac{\beta}{\delta}\tau_m\Phi(P_\tau l_{\vert_{I_m}},d_{\tau,m})+\int\limits_{I_m} f(t)dt\right).
\end{equation}
$F_m\colon\Lzwo\to\Lzwo$ is well-defined as $\Phi$ maps into $\Hzwei\cap\Heinsnull\subset\Lzwo$ and due to the definition of the Bochner integral. Moreover, $F_m$ is a contraction for all $m=1,...,M$ independent of the size of $\tau$ as we will see in the following: Let $p,q\in\Lzwo$ be given. We obtain 
\begin{align*}
\Vert F_m(p)- F_m(q)\Vert &\leq \frac{\frac{\beta}{\delta}\tau_m}{1+\frac{\beta}{\delta}\tau_m}\Vert \Phi(P_\tau l_{\vert_{I_m}},p)-\Phi(P_\tau l_{\vert_{I_m}},q)\Vert \\
&\leq \frac{\frac{\beta}{\delta}\tau_m}{1+\frac{\beta}{\delta}\tau_m}\Vert p-q\Vert.
\end{align*}
Note that the last inequality follows from the Lipschitz continuity of $\Phi\colon\Lzwo\times\Lzwo\to\Lzwo$ with respect to $d$, i.e. it holds 
\[\Vert \Phi(l,p)-\Phi(l,q)\Vert\leq \Vert p-q\Vert\]
with Lipschitz constant $1$.
Thus, $F_m$ is a contraction and Banach's fixed point theorem yields a unique solution. The uniqueness of $\dt$ on each subinterval then gives the uniqueness of $\vt$ on each subinterval and consequently the uniqueness of both functions on the whole time horizon. The regularity of $d_{\tau,m}$ solely relies on the regularity of $d_0$, that is, if we assume that $d_0\in\Hzwei$ and $f\in L^2(\I;\Hzwei)$ then $$d_{\tau,m}=\frac{1}{1+\frac{\beta}{\delta}\tau_m}\left(d_{\tau,m-1}+\frac{\beta}{\delta}\tau_m\Phi(P_\tau l_{\vert_{I_m}},d_{\tau,m})+\int\limits_{I_m} f(t)dt\right)\in\Hzwei$$ as well for all $m=1,...,M$ as $\Phi$ maps $\Lzwo\times\Lzwo\to\Hzwei$ and due to the definition of the Bochner integral. 
\end{proof}  
\begin{remark} Due to the dense embedding $X\overset{\text{d}}{\hookrightarrow} \Lzweizwei$ and $d\in C(\I;\Lzwo)$ the continuous solution $(\varphi,d)\in V\times X$ also satisfies the semidiscrete state equation. Therefore, we have the property of Galerkin orthogonality
\begin{equation} \label{eq:galerkinorthotime}
\B((\varphi-\varphi_\tau,d-d_\tau),(\psi,\lambda))=0 \quad \forall (\psi,\lambda)\in \Vtau\times\Xtau.
\end{equation}
\end{remark}
The first step in proving a priori error estimates is the derivation of stability estimates for the solution of the semidiscrete state equation \eqref{eq:semidiscretestatelinear}. We will establish them in several steps and start with an estimate for $\Vert\dt\Vert_{\infty,2}$:
\begin{lemma} \label{lem:stabestimateinftyzwo}
For the solution $(\vt,\dt)\in\Vtau\times\Xtau$ of the semidiscrete state equation \eqref{eq:semidiscretestatelineargeneralized} with right-hand sides $l,f\in\Leinszwei$ and initial state $d_0\in\Lzwo$ the stability estimate
\begin{equation}
\label{eq:stabestimateinftytwo}
\Vert\dt\Vert_{\infty,2}\leq C\{\Vert d_0\Vert+\Vert l\Vert_{\Leinszwei}+\Vert f\Vert_{\Leinszwei}\}
\end{equation}
holds true with a constant $C>0$ independent of $\tau$.
\end{lemma}
\begin{proof}
At first, by choosing the test functions to vanish outside of the subinterval $I_m$, we reduce the semidiscrete problem onto one subinterval
\begin{align} \label{eq:semidiscretestateoneinterval}
\alpha(\nabla \varphi_\tau,\nabla \psi)_{I_m\times\Omega}-\beta (d_\tau-\varphi_\tau,\psi)_{I_m\times\Omega} 
+&\frac{\beta}{\delta}(d_\tau-\varphi_\tau,\lambda)_{I_m\times\Omega}+ (d_{\tau,m},\lambda_m) \\
&=(l,\psi)_{I_m\times\Omega}+(d_{\tau,m-1},\lambda_m)+(f,\lambda)_{I_m\times\Omega} \nonumber
\end{align} 
for $m=1,...,M$ starting with $d_{\tau,0}:=d_0$. Next, by choosing $\psi=0$ and $\lambda_m=d_{\tau,m}$ we arrive at
\[\frac{\beta}{\delta}\tau_m\Vert d_{\tau,m}\Vert^2+\Vert d_{\tau,m}\Vert^2 \leq \Vert d_{\tau,m-1}\Vert\Vert d_{\tau,m}\Vert+\frac{\beta}{\delta}\tau_m\Vert\varphi_{\tau,m}\Vert\Vert d_{\tau,m}\Vert+\Vert\int\limits_{I_m} f(t)dt\Vert\Vert d_{\tau,m}\Vert.\]
An estimate for $\Vert\varphi_{\tau,m}\Vert$ can be established by choosing $\psi=\varphi_{\tau,m}$ and $\lambda=0$. This gives us
\begin{align*}
\beta\tau_m\Vert\varphi_{\tau,m}\Vert^2&\leq (l,\varphi_{\tau,m})_{I_m\times\Omega}+\beta\tau_m (d_{\tau,m},\varphi_{\tau,m}) \\
&\leq \Vert\int\limits_{I_m} l(t)dt\Vert\Vert\varphi_{\tau,m}\Vert + \beta\tau_m\Vert d_{\tau,m}\Vert\Vert\varphi_{\tau,m}\Vert.
\end{align*}
Division by $\beta\tau_m\Vert\varphi_{\tau,m}\Vert$ then leads to
\begin{equation} \label{eq:zwischenergvarphitm}
\Vert\varphi_{\tau,m}\Vert\leq \frac{1}{\beta}\Vert P_\tau l_{\vert I_m}\Vert+\Vert d_{\tau,m}\Vert.
\end{equation}
Now, the estimate for $\Vert d_{\tau,m}\Vert$ may be continued by
\begin{align*}
(1+\frac{\beta}{\delta}\tau_m) \Vert d_{\tau,m}\Vert & \leq \Vert d_{\tau,m-1}\Vert + \frac{\beta}{\delta}\tau_m\left(\frac{1}{\beta}\Vert P_\tau l_{\vert I_m}\Vert+\Vert d_{\tau,m}\Vert\right)+\Vert\int\limits_{I_m} f(t)dt\Vert \\
\Rightarrow \Vert d_{\tau,m}\Vert&\leq \Vert d_{\tau,m-1}\Vert + \frac{1}{\delta} \Vert l\Vert_{L^1(I_m;\Lzwo)}+\Vert f\Vert_{L^1(I_m;\Lzwo)}.
\end{align*} 
Induction then leads to
\begin{align*}
\Vert d_{\tau,m}\Vert&\leq \Vert d_0\Vert +\frac{1}{\delta}\Vert l\Vert_{L^1((0,t_m];\Lzwo)}+\Vert f\Vert_{L^1((0,t_m];\Lzwo)}\leq \Vert d_0\Vert +\frac{1}{\delta}\Vert l\Vert_{L^1(I;\Lzwo)}+\Vert f\Vert_{L^1(I;\Lzwo)}
\end{align*}
for all $m=1,...,M$. Thus, we arrive at 
\[\Vert \dt\Vert_{\infty,2}\leq C\{\Vert d_0\Vert +\Vert l\Vert_{L^1(I;\Lzwo)}+\Vert f\Vert_{L^1(I;\Lzwo)}\}. \]
%and the assertion follows with the continuity of the embedding $\Lzweizwei\hookrightarrow L^1(\I;\Lzwo)$. 
\end{proof}
As the embeddings $L^\infty(\I;\Lzwo)\hookrightarrow\Lzweizwei\hookrightarrow\Leinszwei$ are continuous we have the following 
\begin{corollary} \label{col:stabestimatedtnonlinear}
For the solution $(\vt,\dt)\in\Vtau\times\Xtau$ of the semidiscrete state equation \eqref{eq:semidiscretestatelineargeneralized} with right-hand sides $l,f\in\Lzweizwei$ and initial state $d_0\in\Lzwo$ the stability estimate
\begin{equation}
\label{e	q:stabestimatedtwotwo}
\Vert\dt\Vert_{I\times\Omega}\leq C\{\Vert d_0\Vert+\Vert l\Vert_{I\times\Omega}+\Vert f\Vert_{I\times\Omega}\}
\end{equation}
holds true with a constant $C>0$ independent of $\tau$. 
\end{corollary}
We are now able to prove the main theorem regarding stability estimates:
\begin{theorem} \label{thm:stabeststatetimethm}
For the solution $(\varphi_\tau,d_\tau)\in\Vtau\times\Xtau$ of the dG(0) semidiscretized state equation \eqref{eq:semidiscretestatelineargeneralized} with right-hand sides $l,f\in\Lzweizwei$ and initial state $d_0\in \Lzwo$ the stability estimate
\begin{equation} \label{eq:stabeststatetime}
\Vert\Delta \varphi_\tau\Vert^2_{I\times\Omega}+\Vert\nabla\varphi_\tau\Vert^2_{I\times\Omega}+\Vert\varphi_\tau\Vert^2_{I\times\Omega}+\Vert d_\tau\Vert^2_{I\times\Omega}+\sum\limits_{m=1}^M\tau_m^{-1}\Vert[d_\tau]_{m-1}\Vert^2\leq C\{\Vert l\Vert^2_{I\times\Omega}+\Vert f\Vert^2_{I\times\Omega}+\Vert d_0\Vert^2\}
\end{equation}
holds true with a constant $C>0$ independent of $\tau$. The jump term $[d_\tau]_0$ is defined as $d_{\tau,0}^+-d_0$.
\end{theorem}
\begin{proof}
The theorem can be proven along the lines of \cite{MV08}, Thm. 4.1/4.3, to derive stability estimates. The estimate for $\Vert\vt\Vert_{I\times\Omega}$ follows by choosing $\psi=\vt$ and $\lambda=0$ in \eqref{eq:semidiscretestateoneinterval} and summing up over all $m=1,...,M$ together with the just derived estimate for $\Vert\dt\Vert^2_{I\times\Omega}$. The estimate for $\Delta\vt$ and $\nabla\vt$ follows after partial integration in space with the choice $\psi=-\Delta\vt$ and $\lambda=0$. Last, the estimate for the jump terms may be concluded by choosing $\psi=0$ and $\lambda=[\dt]_{m-1}$. 
\end{proof}
The above result is also applicable to dual equations. For given right-hand sides $g_1,g_2\in\Lzweizwei$ and a given terminal state $p_T\in\Lzwo$ the corresponding dual equation is given as
\begin{equation} \label{eq:semidiscretedualeq}
\B((\psi,\lambda),(\zt,\pt))= (g_1,\psi)_{I\times\Omega}+(g_2,\lambda)_{I\times\Omega}+(p_T,\lambda^-_M) \quad \forall (\psi,\lambda)\in \Vtau\times\Xtau
\end{equation}
Note, that the semidiscrete bilinear form $\B$ can equivalently be expressed as
\begin{align}  \label{eq:semidiscreteBdual}
\B((\psi,\lambda),(\zt,\pt))&=\alpha(\nabla \zt,\nabla \psi)_{I\times\Omega}+ (\beta\zt-\frac{\beta}{\delta}\pt,\psi-\lambda)_{I\times\Omega} \\
&-\sum\limits_{m=1}^M (\partial_t \pt,\lambda)_{I_m\times\Omega}-\sum\limits_{m=1}^{M-1}([\pt]_m,\lambda^-_m)+(p_{\tau,M}^-,\lambda^-_M) \nonumber
\end{align}
for $(\zt,\pt),(\psi,\lambda)\in\Vtau\times\Xtau$.
\begin{corollary} \label{cor:stabestdualtime}The semidiscrete dual equation \eqref{eq:semidiscretedualeq} possesses a unique solution $(\zt,\pt)\in\Vtau\times\Xtau$ for all $g_1,g_2\in\Lzweizwei$ and $p_T\in\Lzwo$. Moreover, the stability estimate
\begin{equation} \label{eq:stabestdualtime}
\Vert \Delta \zt\Vert^2_{I\times\Omega}+\Vert\nabla \zt\Vert^2_{I\times\Omega}+\Vert \zt\Vert^2_{I\times\Omega}+ \Vert \pt\Vert^2_{I\times\Omega}+\sum\limits_{m=1}^M \tau_m^{-1}\Vert [\pt]_m\Vert^2 \leq C\{\Vert g_1\Vert^2_{I\times\Omega}+\Vert g_2\Vert^2_{I\times\Omega}+\Vert p_T\Vert^2\}
\end{equation}
holds true with a constant $C>0$ independent of $\tau$. The jump term $[p_\tau]_M$ is defined as $p_T-p_{\tau,M}^-$.
\end{corollary}
We now turn our attention to the a priori error estimates for the temporal discretization error. We split the temporal errors as
\[e_\tau^d=d-\dt=\underbrace{d-\pi_\tau d}_{= \eta_\tau^d}+\underbrace{\pi_\tau d-\dt}_{= \xi_\tau^d} \qquad e_\tau^\varphi=\varphi-\vt=\underbrace{\varphi-P_\tau \varphi}_{= \eta_\tau^\varphi}+\underbrace{P_\tau\varphi-\vt}_{= \xi_\tau^\varphi}. \]
The course of the proof of temporal a priori error estimates is similar to the steps taken in \cite{MV08}. We first prove the boundedness of the error by the interpolation and projection errors, respectively, and apply known error estimates for the interpolation and projection operators afterwards. 
\begin{lemma} \label{lem:errortimefirstlemma}
For the projection errors $\eta_\tau^\varphi,\eta_\tau^d$ the equality
\[\B((\etavt,\etadt),(\psi,\lambda))=-\beta(\etadt,\psi)_{I\times\Omega}+\frac{\beta}{\delta}(\etadt,\lambda)_{I\times\Omega}\]
\end{lemma}
holds true for all $(\psi,\lambda)\in\Vtau\times\Xtau$.
\begin{proof}
Similar to \cite{MV08}, Lem. 5.2, we use representation \eqref{eq:semidiscreteBdual} for the semidiscrete bilinear form. Due to the definition of $\pi_\tau d$ we have $\eta_{\tau,m}^{d,-}=0$ for all $m=1,...,M$. Thus all these terms as well as the terms containing temporal derivatives vanish. The $\Lzweizwei$-scalar products containing $\etavt$ are zero for test functions $\psi\in\Vtau\subset\Xtau$ and $\lambda\in\Xtau$ due to the definition of $P_\tau$ (see \eqref{eq:nablaPtvarphi}). 
\end{proof}
\begin{lemma} \label{lem:timeerrorvseta}
The discretization error is bounded by the projection error, that is
\[\Vert \evt\Vert_{I\times\Omega}+ \Vert\edt\Vert_{I\times\Omega} \leq C\{\Vert \etavt\Vert_{I\times\Omega}+\Vert\etadt\Vert_{I\times\Omega}\}\]
holds true. 
\end{lemma}
\begin{proof}
The lemma can be proven  following the arguments used in \cite{MV08}, Lem. 5.3. We consider the dual equation 
\[\B((\psi,\lambda),(\zt,\pt))=(\evt,\psi)_{I\times\Omega}+(\edt,\lambda)_{I\times\Omega} \quad \forall (\psi,\lambda)\in\Vtau\times\Xtau.\]
If we set $p_T=0$ this equation has a unique solution $(\zt,\pt)\in\Vtau\times\Xtau$. Then the following representation holds due to the Galerkin orthogonality and Lemma \ref{lem:errortimefirstlemma}
\begin{align*}
(\evt,\evt)_{I\times\Omega}+(\edt,\edt)_{I\times\Omega}&= (\xivt,\evt)_{I\times\Omega}+(\etavt,\evt)_{I\times\Omega}+(\xidt,\edt)_{I\times\Omega}+(\etadt,\edt)_{I\times\Omega} \\
&=(\etavt,\evt)_{I\times\Omega}+(\etadt,\edt)_{I\times\Omega}+\B((\xivt,\xidt),(\zt,\pt)) \\
&=(\etavt,\evt)_{I\times\Omega}+(\etadt,\edt)_{I\times\Omega}-\B((\etavt,\etadt),(\zt,\pt)) \\
&=(\etavt,\evt)_{I\times\Omega}+(\etadt,\edt)_{I\times\Omega}+\beta(\etadt,\zt)_{I\times\Omega}-\frac{\beta}{\delta}(\etadt,\pt)_{I\times\Omega}.
\end{align*}
Cauchy-Schwarz's inequality and the stability estimates for dual solutions then lead to
\[\Vert\evt\Vert^2_{I\times\Omega}+\Vert\edt\Vert^2_{I\times\Omega}\leq C(\Vert\etavt\Vert_{I\times\Omega}+\Vert\etadt\Vert_{I\times\Omega})(\Vert\evt\Vert_{I\times\Omega}+\Vert\edt\Vert_{I\times\Omega}).\]

\end{proof}
This gives us the main result of this section. Due to Assumption \ref{ass:regularityvarphi} $\varphi$ has the required regularity. 
\begin{theorem} \label{thm:temporalerrorestimate}
Let $l\in\Lzweizwei, d_0\in\Lzwo$ and Assumption \ref{ass:regularityvarphi} be fulfilled. For the errors $\evt:=\varphi-\vt$ and $\edt:=d-\dt$ between the continuous solutions $(\varphi,d)\in V\times X$ of \eqref{eq:weakformcont} and the dG(0) semidiscretized solutions $(\vt,\dt)\in\Vtau\times\Xtau$ of \eqref{eq:semidiscretestatelinear}, we have the error estimate
\[\Vert \evt\Vert_{I\times\Omega}+\Vert \edt\Vert_{I\times\Omega}\leq C\tau\{\Vert \partial_t\varphi\Vert_{I\times\Omega}+\Vert\partial_t d\Vert_{I\times\Omega}\}\]
with a constant $C$ independent of the temporal discretization parameter $\tau$. 
\end{theorem}
\begin{proof}
It suffices to bound the interpolation error $\Vert\etadt\Vert_{I\times\Omega}$ and the projection error $\Vert\etavt\Vert_{I\times\Omega}$. It is well known that
\[\Vert\etadt\Vert_{I_m\times\Omega}\leq C\tau_m\Vert\partial_t d\Vert_{I_m\times\Omega}\]
holds true for $d$. Due to Assumption \ref{ass:regularityvarphi} we know that $\varphi$ is continuous in time. Therefore, $\pi_\tau$ is applicable to $\varphi$, maps into $\Vtau$ and we have the same error estimate as above, namely
\[\Vert\varphi-\pi_\tau\varphi\Vert_{I_m\times\Omega}\leq C\tau_m\Vert\partial_t \varphi\Vert_{I_m\times\Omega}.\]
The estimate for $\etavt$ then follows due to the best approximation property of the $L^2$-projection. Now, 
\begin{align*}
\Vert \evt\Vert^2_{I\times\Omega}+\Vert \edt\Vert^2_{I\times\Omega}&\leq C\{\Vert \etavt\Vert_{I\times\Omega}^2+\Vert\etadt\Vert_{I\times\Omega}^2\} = C\sum\limits_{m=1}^M\{\Vert \etavt\Vert_{I_m\times\Omega}^2+\Vert\etadt\Vert_{I_m\times\Omega}^2\} \\
&\leq C\sum\limits_{m=1}^M \tau_m^2\{\Vert\partial_t d\Vert_{I_m\times\Omega}^2+\Vert\partial_t \varphi\Vert_{I_m\times\Omega}^2\} \leq C\tau^2\{\Vert\partial_t d\Vert_{I\times\Omega}^2+\Vert\partial_t \varphi\Vert_{I\times\Omega}^2\}
\end{align*}
finishes the proof.
\end{proof}
\subsection{Discretization in space}
In this paragraph we deal with the space-time discretized state equation which was given by \eqref{eq:fullydiscretestate}. As we will need them later, we briefly summarize results well known for the discretization in space of the PDE.
\begin{lemma} \label{lem:spatialerrorvh}
Let $l,d\in\Lzwo$ be given. Then the variational problem
\begin{equation}
\label{eq:spacediscretestatevarphi}
\alpha(\nabla\varphi_h,\nabla\psi)+\beta(\varphi_h,\psi)=(\beta d+l,\psi) \qquad \forall \psi\in \Vh
\end{equation}
possesses a unique solution $\varphi_h\in\Vh$ and the solution operator $\Phi_h\colon\Lzwo\times\Lzwo\to\Vh, \Phi_h(l,d)=\varphi_h$, is Lipschitz continuous with the same Lipschitz constant as its counterpart $\Phi$. Furthermore, the spatial error between the solution $\Phi(l,d)$ of the continuous equation \eqref{eq:weakvarphicont} and the discrete solution $\Phi_h(l,d)$ of the discrete equation \eqref{eq:spacediscretestatevarphi} is of order $h^2$, that is there exists a constant $C>0$ independent of $h$, such that
\begin{equation}
\label{eq:spatialerrorvarphi}
\Vert \Phi(l,d)-\Phi_h(l,d)\Vert \leq Ch^2\Vert\nabla^2 \Phi(l,d)\Vert
\end{equation}
holds true. 
\end{lemma}
As the existence of a unique $\vth(t)=\sum\limits_{m=1}^M\Phi_h(P_\tau l_{\vert I_m},d_{\tau h,m})\chi_{I_m}(t)$ is known from the above lemma for every $l\in\Lzweizwei$ and $\dth\in\Xth$ the existence of a unique solution $(\vth,\dth)\in\Vth\times\Xth$ follows as in the semidiscrete case by applying Banach's fixed point theorem to the fixed point problem in $\Xh$
\begin{equation}
\label{eq:discreteODEreduced}
d_{\tau h,m}=d_{\tau h,m-1}-\frac{\beta}{\delta}\tau_m(d_{\tau h,m}-\Phi_h(P_\tau l_{\vert I_m},d_{\tau h,m}))
\end{equation}
for every $m=1,...,M$ starting with $d_{\tau h,0}:=P^X_h(d_0)$. This gives us
\begin{lemma}\label{lem:existencediscretestateeq}
Let $l\in\Lzweizwei$ and $d_0\in\Lzwo$ be given. Then, the discrete state equation \eqref{eq:fullydiscretestate} possesses a unique solution $(\vth,\dth)\in\Vth\times\Xth$ . 
\end{lemma}
Moreover, because of $\Vth\subset \Vtau, \Xth\subset\Xtau$ we directly have the spatial Galerkin orthogonality of the error $(\evh,\edh)=(\vt-\vth,\dt-\dth)$, that is
\begin{equation} \label{eq:Galerkinorthospace}
\B((\evh,\edh),(\psi,\lambda))=0 \quad \forall (\psi,\lambda)\in\Vth\times\Xth.
\end{equation}
As in the last paragraph we split the spatial errors $\evh=\vt-\vth,\edh=\dt-\dth$ as 
\[\evh=\vt-\vth=\underbrace{\vt-\pi^V_h\vt}_{=\etavh}+\underbrace{\pi^V_h\vt-\vth}_{=\xivh},\] 
\[\edh=\dt-\dth=\underbrace{\dt-\pi^X_h\dt}_{=\etadh}+\underbrace{\pi^X_h\dt-\dth}_{=\xidh}.\]
We also require the dual solution $(\zth,\pth)\in\Vth\times\Xth$ of the following problem:
\begin{equation} \label{eq:dualspacialerror}
\B((\psi,\lambda),(\zth,\pth))=(g_1,\psi)_{I\times\Omega}+(g_2,\lambda)_{I\times\Omega} \quad \forall (\psi,\lambda)\in\Vth\times\Xth
\end{equation}
with the terminal condition $p_T=0$ and $g_1=\evh$, $g_2=\edh$. \\

We will derive spatial error estimates with the same technique used for the temporal error. In particular, we need stability estimates for space-time discrete solutions. Fortunately, the results of Theorem \ref{thm:stabeststatetimethm} and Corollary \ref{cor:stabestdualtime} also hold true for space-time discrete solutions with only minor changes. We replace the Laplacian with its discrete counterpart $\Delta_h\colon \Vh\to\Vh$ defined via
\[(\Delta_h\varphi,\psi)=-(\nabla\varphi,\nabla\psi) \quad \forall \psi\in\Vh\]
and project the initial and final state onto $\Xh$ by means of the $L^2$-projection $P^X_h$. For the convenience of the reader, we state the estimates for the space-time discrete solutions.
\begin{theorem}
\label{thm:stabestimatesspace}
For the solution $(\vth,\dth)\in\Vth\times\Xth$ of the dG(0)cG(1) discretized state equation \eqref{eq:fullydiscretestate} with right-hand side $l\in\Lzweizwei$ and initial state $d_0\in\Lzwo$ the stability estimate
\begin{equation}
\label{eq:stabestimatespace}
\Vert\Delta_h\vth\Vert_{I\times\Omega}^2+\Vert\nabla\vth\Vert^2_{I\times\Omega}+\Vert\vth\Vert^2_{I\times\Omega}+\Vert\dth\Vert^2_{I\times\Omega}
+\sum\limits_{m=1}^M\tau_m^{-1}\Vert [\dth]_{m-1}\Vert^2\leq C\{\Vert l\Vert^2_{I\times\Omega}+\Vert P^X_hd_0\Vert^2\}
\end{equation}
holds true with a constant $C>0$ independent of $\tau$ and $h$. The jump term $[\dth]_0$ is defined as $d_{\tau h,0}^+-P^X_hd_0$.
\end{theorem}
\begin{corollary}
\label{cor:stabestimatesspacedual}
For the solution $(\zth,\pth)\in\Vth\times\Xth$ of the discrete dual state equation \eqref{eq:dualspacialerror} with right-hand sides $g_1,g_2\in\Lzweizwei$ and final state $p_T\in\Lzwo$ the stability estimate
\begin{align}
\label{eq:stabestimatespacedual}
\Vert\Delta_h\zth\Vert_{I\times\Omega}^2+\Vert\nabla\zth\Vert^2_{I\times\Omega}+\Vert\zth\Vert^2_{I\times\Omega}+\Vert\pth\Vert^2_{I\times\Omega}
+&\sum\limits_{m=1}^M\tau_m^{-1}\Vert [\pth]_{m}\Vert^2 \\&\leq C\{\Vert g_1\Vert^2_{I\times\Omega}+\Vert g_2\Vert^2_{I\times\Omega}+\Vert P^X_hp_T\Vert^2\} \nonumber
\end{align}
holds true with a constant $C>0$ independent of $\tau$ and $h$. The jump term $[\pth]_M$ is defined as $P^X_hp_T-p_{\tau h,M}^-$.
\end{corollary}
Similar to Lemma \ref{lem:errortimefirstlemma} we have the following result
\begin{lemma} \label{lem:projectioninbilinearformlinear}
For $(\psi,\lambda)\in\Vth\times\Xth$ we have the following identity for the projection errors
\begin{equation} \label{eq:projectionspace1}
\B((\etavh,\etadh),(\psi,\lambda))=\alpha (\nabla\etavh,\nabla\psi)_{I\times\Omega}-\frac{\beta}{\delta}(\etavh,\lambda)_{I\times\Omega}
\end{equation}
%\begin{equation} \label{eq:projectionspace2}
%\B((\psi,\lambda),(\etazh,\etaph))=\alpha (\nabla\etazh,\nabla\psi)_I-\beta(\etazh,\lambda)_I
%\end{equation}
\end{lemma}
\begin{proof}
The assertion follows directly if we insert the projections errors into the original formulation of the bilinear form
% and the dual formulation \eqref{eq:semidiscreteBdual} respectively 
and make use of the definition of $\pi_h$. Note, that the scalar product containing $\etavh$ 
%and $\etazh$
and a test function in $\lambda\in\Xth$ does not vanish as the projection error is orthogonal only to the subspace $\Vth\subset\Xth$. 
\end{proof}
Next, we prove the boundedness of the errors by certain projection errors.
\begin{lemma} \label{lem:boundednesserrorspatial}
There exists a constant $C>0$ independent of $\tau$ and $h$ such that the discretization errors are bounded by projection errors, that is we have
\begin{equation}
\label{eq:boundednesserrorspatial}
\Vert\evh\Vert_{I\times\Omega}+\Vert\edh\Vert_{I\times\Omega}\leq C\{\Vert\etavh\Vert_{I\times\Omega}+\Vert\etadh\Vert_{I\times\Omega}+\Vert\vt-\rho_h\vt\Vert_{I\times\Omega}
\}.
\end{equation}
\end{lemma} 
\begin{proof}
Let $(\zth,\pth)\in\Vth\times\Xth$ be the solution of the discrete dual equation \eqref{eq:dualspacialerror} with right-hand sides $g_1=\evh$ and $g_2=\edh$ and terminal state $p_T=0$. By employing the spatial Galerkin-orthogonality we have the following representation of the error
\begin{align*}
(\evh,\evh)_{I\times\Omega}+(\edh,\edh)_{I\times\Omega}&=(\xivh,\evh)_{I\times\Omega}+(\etavh,\evh)_{I\times\Omega}+(\xidh,\edh)_{I\times\Omega}+(\etadh,\edh)_{I\times\Omega} \\
&= (\etavh,\evh)_{I\times\Omega}+(\etadh,\edh)_{I\times\Omega}+\B((\xivh,\xidh),(\zth,\pth)) \\
&=(\etavh,\evh)_{I\times\Omega}+(\etadh,\edh)_{I\times\Omega}-\B((\etavh,\etadh),(\zth,\pth)) \\
&=(\etavh,\evh)_{I\times\Omega}+(\etadh,\edh)_{I\times\Omega}-\alpha(\nabla\etavh,\nabla\zth)_{I\times\Omega}+\frac{\beta}{\delta}(\etavh,\pth)_{I\times\Omega} \\
&=(\etavh,\evh)_{I\times\Omega}+(\etadh,\edh)_{I\times\Omega}-\alpha(\nabla(\vt-\pi^V_h\vt),\nabla\zth)_{I\times\Omega}+\frac{\beta}{\delta}(\etavh,\pth)_{I\times\Omega} \\
\end{align*}
Next, we apply the Ritz-projection to $\vt$. Afterwards, we are allowed to use the definition of the discrete Laplacian. This leads to
\begin{align*}
&(\evh,\evh)_{I\times\Omega}+(\edh,\edh)_{I\times\Omega}=\\
&(\etavh,\evh)_{I\times\Omega}+(\etadh,\edh)_{I\times\Omega}+\alpha(\rho_h\vt-\pi^V_h\vt,\Delta_h\zth)_{I\times\Omega}+\frac{\beta}{\delta}(\etavh,\pth)_{I\times\Omega} \\
&\leq C(\Vert\etavh\Vert_{I\times\Omega}+\Vert\etadh\Vert_{I\times\Omega})(\Vert\evh\Vert_{I\times\Omega}+\Vert\edh\Vert_{I\times\Omega}+\Vert\pth\Vert_{I\times\Omega})
\\ &+\alpha(\Vert\etavh\Vert_{I\times\Omega}+\Vert\vt-\rho_h\vt\Vert_{I\times\Omega})\Vert\Delta_h\zth\Vert_{I\times\Omega}.
\end{align*} 
Finally, with the stability estimates of Corollary \ref{cor:stabestimatesspacedual} we obtain
\[(\evh,\evh)_{I\times\Omega}+(\edh,\edh)_{I\times\Omega}\leq C(\Vert\etavh\Vert_{I\times\Omega}+\Vert\etadh\Vert_{I\times\Omega}+\Vert\vt-\rho_h\vt\Vert_{I\times\Omega})(\Vert\evh\Vert_{I\times\Omega}+\Vert\edh\Vert_{I\times\Omega}).\]
\end{proof}
We are now in the position to prove our main result regarding the spatial error. The required regularity for $\dt$ is assured if we assume $d_0\in\Hzwei$. 
\begin{theorem} \label{thm:spacialerrorestimate}
Let $l\in\Lzweizwei$ and $d_0\in\Hzwei$ be given. For the errors $\evh=\vt-\vth$ and $\edh=\dt-\dth$ between the dG(0) semidiscretized solutions $\vt\in\Vtau,\dt\in\Xtau$ of \eqref{eq:semidiscretestatelinear} and the fully dG(0)cG(1) discretized solutions $\vth\in\Vth,\dth\in\Xth$ of \eqref{eq:fullydiscretestate}, we have the error estimate
\[\Vert \evh\Vert_{I\times\Omega}+\Vert \edh\Vert_{I\times\Omega}\leq Ch^2\{\Vert\nabla^2\vt\Vert_{I\times\Omega}+\Vert\nabla^2\dt\Vert_{I\times\Omega}\},\]
where the constant $C$ is independent of the mesh size $h$ and the temporal discretization parameter $\tau$. 
\end{theorem}
\begin{proof}
Based on the previous results and due to the definition of $\pi^V_h,\pi^X_h$ and $\rho_h$ the estimate directly follows from the approximation properties of the $L^2$-projections $P^V_h,P^X_h$ and the Ritz-projection $R_h$. 
\end{proof}
To summarize the main results of this section, we state the overall error estimate:
\begin{theorem}
\label{thm:overallerrorestimate}
Let $l\in\Lzweizwei, d_0\in\Hzwei$ and Assumption \ref{ass:regularityvarphi} be fulfilled. For the errors $\varphi-\vth$ and $d-\dth$ between the continuous solutions $\varphi\in V,d\in X$ of \eqref{eq:weakformcont} and the fully dG(0)cG(1) discretized solutions $\vth\in\Vth,\dth\in\Xth$ of \eqref{eq:fullydiscretestate}, we have the error estimate
\begin{align*}
\Vert \varphi-\vth\Vert_{I\times\Omega}+\Vert d-\dth\Vert_{I\times\Omega}&\leq 
C\tau\{\Vert \partial_t\varphi\Vert_{I\times\Omega}+\Vert\partial_t d\Vert_{I\times\Omega}\} \\
&+  Ch^2\{\Vert\nabla^2\vt\Vert_{I\times\Omega}+\Vert\nabla^2\dt\Vert_{I\times\Omega}\},\end{align*}
where the constant $C$ is independent of the mesh size $h$ and the temporal discretization parameter $\tau$. 
\end{theorem}
\section{Error estimates for the associated optimal control problem}
In this section we turn our attention towards an associated optimal control problem and its discretization. To be more specific, we want to estimate the error between the continuous optimal control and its space-time discretization of the following optimal control problem governed by our PDE-ODE-system \eqref{mod:linmodbegin}-\eqref{mod:linmodend}
\begin{equation} \label{eq:linoptcontprob}
\min J(\varphi,d,l)=\frac{1}{2}\Vert\varphi-\varphi_d\Vert_{I\times\Omega}^2+\frac{1}{2}\Vert d-d_d\Vert_{I\times\Omega}^2+\frac{\alpha_l}{2}\Vert l\Vert^2_{I\times\Omega}
\end{equation}
with $(\varphi_d,d_d)\in\Lzweizwei^2$ being two given desired states, $\alpha_l>0$ being a regularization parameter and $(\varphi,d)\in V\times X$ being the weak solution of \eqref{mod:linmodbegin}-\eqref{mod:linmodend} for the right hand side $l\in\Lzweizwei$. 
Before discretizing the control we first have a look at the optimal control problem at the continuous, the semidiscretized and space-time discretized level. All these levels of discretization will be referred to a \textit{variational approaches} in accordance with \cite{MH05}. The case with a discrete control will be referred to as the \textit{fully discretized problem}.
\subsection{The optimal control problem on different levels of discretization}
First, we define the control-to-state operator $S\colon\Lzweizwei\to V\times X$ given by $S(l)=(\varphi,d)=(S_1(l),S_2(l))$. Similar to \cite{S17}, Rem. 3.24, and by means of \cite{S17}, Lem. 5.10, it can be shown, that \\ $S_2 \colon \Lzweizwei \to X$ is Lipschitz continuous. Thus, $S$ is Lipschitz continuous as well. Using the control-to-state operator we can reduce the optimal control problem to the control variable
\begin{equation} \label{eq:reducedoptimalcontrolprob}
\min  j(l):=J(S(l),l) \text{ subject to } l\in\Lzweizwei.
\end{equation}
\begin{theorem} \label{thm:existenceoptimalcontrol}
For given desired states $\varphi_d,d_d\in\Lzweizwei$, initial value $d_0\in\Lzwo$ and $\alpha_l>0$ the optimal control problem \eqref{eq:reducedoptimalcontrolprob} admits a unique solution $\ovl\in\Lzweizwei$ with corresponding state $(\ovv,\ovd)\in V\times X$. 
\end{theorem}
\begin{proof}
See \cite{FT09}, Thm. 2.16 for the proof of existence and uniqueness. 
\end{proof}
The necessary optimality condition is given as
\begin{equation} \label{eq:fonlinearprimal}
j^\prime(\overline{l})(\delta l)=\partial_{(\varphi,d)}J(S(l),l)S^\prime(l)(\delta l)+\partial_l J(S(l),l)(\delta l)=0 \qquad \forall \delta l\in\Lzweizwei.
\end{equation}
The derivative of the control-to-state operator $S$ in direction $\delta l$ is the solution of the same variational problem with initial value $\delta d(0)=0$ and right-hand side $\delta l$, that is $S^\prime(l)(\delta l)=(\delta\varphi,\delta d)\in V\times X$ solves
\begin{equation} \label{eq:derivativectslinear}
B((\delta\varphi,\delta d),(\psi,\lambda))=(\delta l,\psi)_{I\times\Omega} \quad \forall (\psi,\lambda)\in V\times X.
\end{equation}
In particular, $S^\prime$ is independent of $l$. \\
Note, that due to the convexity of our problem the first order necessary optimality condition is also sufficient.  Next, we define the adjoint state $(z,p)\in V\times X$ as the solution of the variational problem 
\begin{equation} \label{eq:adjointvariationalform}
B((\psi,\lambda),(z,p))=(\varphi-\varphi_d,\psi)_{I\times\Omega}+(d-d_d,\lambda)_{I\times\Omega} \quad \forall (\psi,\lambda)\in V\times X
\end{equation}
with $(\varphi,d)\in V\times X$ being the solution of the forward problem. The existence of a unique adjoint state can be proven with exactly the same arguments used to show existence of a unique solution of the primal problem. The ODE running backwards in time can be transformed via $\rho(t)=T-t$ into an ODE running forward in time with initial value $\tilde{p}(0)=0$. 
Using the adjoint equation we can reformulate the first order necessary optimality condition as
\begin{equation}
\label{eq:fonlineardual}
j^\prime(\overline{l})(\delta l)=(\alpha_l\overline{l}+\overline{z},\delta l)_{I\times\Omega}=0 \qquad\forall \delta l\in\Lzweizwei.
\end{equation}
In particular, this means that $\ovl=-\frac{1}{\alpha_l}\ovz\in V$. \\
Before we proceed to the semidiscretized level we will first have a look at the temporal regularity of the optimal state $\ovv$. Recall, that we had to assume $\varphi\in\Heinszwei$ in the last section. We will now prove that the optimal state $\ovv$ indeed possesses this regularity if the desired states are more regular in time. %and the Tichonov-parameter $\alpha_l$ is sufficiently large. 
The proof of the next theorem relies on a result for absolute continuity of a function $f\in\Lzweizwei$.
\begin{lemma} \label{lem:absolutecontinuity}
Let $g\colon[\I]\to\Lzwo$ be absolutely continuous on $[0,T]$. If for a function $f\in\Lzweizwei$
\[\Vert f(t_1)-f(t_2)\Vert\leq C\Vert g(t_1)-g(t_2)\Vert \qquad \text{f.a.a } t_1,t_2\in(\I)\]
then $f$ is absolutely continuous on $[\I]$ as well. 
\end{lemma}
\begin{proof}
The assertion can be proven by standard arguments. 
\end{proof}
\begin{theorem}
Let $\varphi_d,d_d\in\Heinszwei$ be two given desired states and let $\ovl\in\Lzweizwei$ be the optimal control for the problem \eqref{eq:reducedoptimalcontrolprob} with associated optimal state $(\ovv,\ovd)\in V\times X$ and adjoint state $(\ovz,\ovp)\in V\times X$ . Then we have the improved regularity
\[\ovv,\ovz,\ovl\in\Heinszwei\]
%\textcolor{blue}{if the Tichonov-parameter $\alpha_l$ fulfills $\alpha_l>\beta^{-2}$}.
\end{theorem}
%\textcolor{blue}{Since $\beta$ is a penalty parameter in the original formulation of the damage model and thus probably larger than $1$, the restriction of $\alpha_l$ is not a problem in applications.} 

\begin{proof}
At first we will show the existence of weak temporal derivatives for $\ovl,\ovv$ and $\ovz$. According to \cite{WM11} a function belongs to $W^{1,1}(\I;\Lzwo)$ (and thus is differentiable almost everywhere) if and only if it is absolutely continuous w.r.t. time.  
We have a look at the optimality system \eqref{eq:weakformcont},\eqref{eq:adjointvariationalform} and \eqref{eq:fonlineardual} pointwisely in time, that is
\begin{align*}
\alpha(\nabla \ovv(t),\nabla \psi)+\beta(\ovv(t),\psi)-\beta(\ovd(t),\psi)+&(\partial_t \ovd(t),\lambda)+\frac{\beta}{\delta}(\ovd(t)-\ovv(t)),\lambda)=(\ovl(t),\psi) \\
\alpha(\nabla \ovz(t),\nabla \omega)+\beta(\ovz(t), \omega-v)-\frac{\beta}{\delta}(\ovp(t),& \omega-v)-(\partial_t \ovp(t),v)\\
&=(\ovv(t)-\varphi_d(t),\omega)+(\ovd(t)-d_d(t),v) \\
(\ovl(t)+\frac{1}{\alpha_l}\ovz(t),\delta l)&=0
\end{align*}
holds true for $\psi,\omega\in\Heinsnull$, $\lambda,v,\delta l\in\Lzwo$ and for almost all $t\in[\I]$.   %Testing the forward problem with $\psi=\ovv(t_1)-\ovv(t_2)$ and $\lambda=0$ yields
%\begin{equation}
%\Vert \ovv(t_1)-\ovv(t_2)\Vert \leq \Vert \ovd(t_1)-\ovd(t_2)\Vert + \frac{1}{\beta}\Vert \ovl(t_1)-\ovl(t_2)\Vert. 
%\end{equation}
%Analogously, we test the adjoint equation with $\omega=\ovz(t_1)-\ovz(t_2), v=0$ to arrive at
%\begin{equation}
%\Vert \ovz(t_1)-\ovz(t_2)\Vert \leq \frac{1}{\delta}\Vert \ovp(t_1)-\ovp(t_2)\Vert + \frac{1}{\beta}\Vert \ovv(t_1)-\ovv(t_2)\Vert+\frac{1}{\beta}\Vert\varphi_d(t_1)-\varphi_d(t_2)\Vert. 
%\end{equation}
%From the variational equation \eqref{eq:fonlineardual} we obtain
%\begin{equation}
%\Vert \ovl(t_1)-\ovl(t_2)\Vert \leq \frac{1}{\alpha_l}\Vert\ovz(t_1)-\ovz(t_2)\Vert. 
%\end{equation}
%The combination of all these estimates as well as the assumption that $\alpha_l>\beta^{-2}$ leads to
%\begin{equation}
%\Vert \ovv(t_1)-\ovv(t_2)\Vert \leq \frac{1}{\alpha_l\beta^2-1}\left(\alpha_l\beta^2\Vert \ovd(t_1)-\ovd(t_2)\Vert+\frac{\beta}{\delta}\Vert \ovp(t_1)-\ovp(t_2)\Vert + \Vert\varphi_d(t_1)-\varphi_d(t_2)\Vert\right). 
%\end{equation}
At first, according to the third equation we have $\ovz(t)=-\alpha_l \ovl(t)$ and are thus allowed to insert this into the second equation. At the same time we choose $\lambda=0$ and $v=0$ as we only work with the elliptic equations. If we consider the system for two different time points $t_1$ and $t_2$ (assuming that all functions exist for these two time points) and subtract the equations which belong together we obtain
\[\alpha(\nabla (\ovv(t_1)-\ovv(t_2)),\nabla \psi)+\beta(\ovv(t_1)-\ovv(t_2),\psi)-\beta(\ovd(t_1)-\ovd(t_2),\psi)=(\ovl(t_1)-\ovl(t_2),\psi)\]
\begin{align*}
-\alpha_l\alpha(\nabla (\ovl(t_1)-\ovl(t_2)),\nabla \omega)&-\alpha_l\beta(\ovl(t_1)-\ovl(t_2),\omega)\\
&=(\ovv(t_1)-\ovv(t_2)-(\varphi_d(t_1)-\varphi_d(t_2)),\omega)+\frac{\beta}{\delta}(p(t_1)-p(t_2),\omega)
\end{align*}
Testing the forward problem with $\psi=\alpha_l(\ovl(t_1)-\ovl(t_2))$, the adjoint equation with $w=\ovv(t_1)-\ovv(t_2)$ and addition of both equations then yields
\begin{align*}
0&=\alpha_l\Vert \ovl(t_1)-\ovl(t_2)\Vert^2 + \Vert \ovv(t_1)-\ovv(t_2)\Vert^2+\alpha_l\beta(\ovd(t_1)-\ovd(t_2),\ovl(t_1)-\ovl(t_2)) \\
&+\frac{\beta}{\delta}(\ovp(t_1)-\ovp(t_2),\ovv(t_1)-\ovv(t_2))-(\varphi_d(t_1)-\varphi_d(t_2),\ovv(t_1)-\ovv(t_2)).
\end{align*}
Next, with Cauchy-Schwarz' inequality and Young's inequality we obtain
\begin{align*}
\alpha_l\Vert \ovl(t_1)-\ovl(t_2)\Vert^2 +\Vert \ovv(t_1)-\ovv(t_2)\Vert^2 \leq \alpha_l\beta\frac{1}{2\varepsilon_1}\Vert \ovd(t_1)-\ovd(t_2)\Vert^2 + \frac{1}{2}\alpha_l\beta\varepsilon_1\Vert \ovl(t_1)-\ovl(t_2)\Vert^2 \\
+ \frac{\beta}{\delta}\frac{1}{2\varepsilon_2}\Vert\ovp(t_1)-\ovp(t_2)\Vert^2 + \frac{\beta}{\delta}\frac{\varepsilon_2}{2}\Vert \ovv(t_1)-\ovv(t_2)\Vert^2 \\
+\frac{1}{2\varepsilon_3}\Vert \varphi_d(t_1)-\varphi_d(t_2)\Vert^2+ \frac{\varepsilon_3}{2}\Vert \ovv(t_1)-\ovv(t_2)\Vert^2.
\end{align*}
We choose $\varepsilon_1=\frac{1}{\beta},\varepsilon_2=\frac{\delta}{2\beta},\varepsilon_3=1$ such that $\alpha_l-\frac{1}{2}\alpha_l\beta\varepsilon_1=\frac{\alpha_l}{2}>0, 1-\frac{\beta}{\delta}\frac{\varepsilon_2}{2}-\frac{\varepsilon_3}{2}=\frac{1}{4}>0$. Finally, we obtain
\begin{equation}
\Vert \ovl(t_1)-\ovl(t_2)\Vert^2 +\Vert \ovv(t_1)-\ovv(t_2)\Vert^2 \leq C\{\Vert \ovd(t_1)-\ovd(t_2)\Vert^2+\Vert\ovp(t_1)-\ovp(t_2)\Vert^2+\Vert \varphi_d(t_1)-\varphi_d(t_2)\Vert^2\}
\end{equation}
with a constant $C>0$ independent of time. This inequality holds true for almost all $t_1,t_2\in [\I]$ with the right-hand side existing for all $t_1,t_2\in[\I]$.
Then, the absolute continuity of $\ovd,\ovp$ and $\varphi_d$ gives the absolute continuity of $\ovv$ and $\ovl$ as thus also of $\ovz$ according to Lemma \ref{lem:absolutecontinuity}. This means, the three functions are almost everywhere differentiable w.r.t. time and we have $\ovv,\ovz,\ovl\in W^{1,1}(\I;\Lzwo)$. In particular, we have the existence of weak temporal derivatives which at least belong to $L^1(\I;\Lzwo)$. \\
In the second part of the proof, we will show, that these weak temporal derivatives actually belong to $\Lzweizwei$. Since a function $y\in W^{1,1}(\I;\Lzwo)$ is differentiable almost everywhere, we have that 
\[\partial_t y(t)=\lim_{h\searrow 0} \frac{y(t+h)-y(t)}{h}\]
has to hold true for almost all $t\in[0,T]$ (see \cite{WM11}, Thm. 3.1.40.). We can use this property to directly compute the temporal derivatives of $\ovl,\ovv$ and $\ovz$ by employing the definition of $\varphi(t)=\Phi(d(t),l(t))$ and $z(t)=\Phi(\frac{1}{\delta}p(t),\varphi(t)-\varphi_d(t))$ and $l(t)=-\frac{1}{\alpha_l}z(t)$. One obtains that the temporal derivatives of $y=\ovv,\ovz,\ovl$ are given as $\partial_t\ovv(t)=\Phi(\partial_t \ovd(t),\partial_t \ovl(t)), \partial_t \ovz(t)=\Phi(\frac{1}{\delta}\partial_t\ovp(t),\partial_t(\ovv(t)-\varphi_d(t)))$ and $\partial_t \ovl(t)=-\frac{1}{\alpha_l}\partial_t \ovz(t)$. As $\Phi$ maps $\Lzwo\times\Lzwo$ to $\Heinsnull\cap\Hzwei$ we have $\partial_t\ovv(t), \partial_t\ovz(t),\partial_t\ovl(t)\in\Heinsnull$ and are thus allowed to choose them as test functions. An estimation analogously to the above one then yields
\[\Vert\partial_t\ovl(t)\Vert^2+\Vert \partial_t\ovv(t)\Vert^2\leq C\left(\Vert \partial_t\ovd(t)\Vert^2+\Vert \partial_t\ovp(t)\Vert^2 + \Vert\partial_t\varphi_d(t)\Vert^2\right). \]
Since all terms on the right-hand side belong to $L^1(\I)$ the same has to hold true for the left-hand side. Thus, we have $\partial_t\ovv,\partial_t\ovl\in\Lzweizwei$ and consequently $\partial_t\ovz\in\Lzweizwei$. This finishes the proof.
\end{proof}
On the semidiscretized level, that is the control is not discretized yet but the states are already discretized in time, we consider the optimal control problem
\begin{equation}
\label{eq:semidiscreteoptimalcontrolprob}
\min j_\tau(l)=J(S_\tau(l),l) \text{ subject to } l\in\Lzweizwei
\end{equation}
with $S_\tau\colon \Lzweizwei \to \Vtau\times\Xtau$ being the solution operator of the semidiscretized state equation \eqref{eq:semidiscretestatelinear}, that is $S_\tau(l)=(\vt,\dt)$. 
\begin{lemma} \label{lem:existenceoptcontrolsemidiscretelevel}
The optimal control problem \eqref{eq:semidiscreteoptimalcontrolprob} admits for $\alpha_l>0$ a unique solution $\overline{l}_\tau\in\Lzweizwei$ with corresponding state $(\overline{\varphi}_\tau,\overline{d}_\tau)\in\Vtau\times\Xtau$.
\end{lemma}
\begin{proof}
The assertion can be proven with exactly the same arguments as in the proof on the continuous level. 
\end{proof}
On this level of discretization the first order optimality condition reads
\begin{equation} \label{eq:fonlinearprimalsemidiscrete}
j_\tau^\prime(\overline{l}_\tau)(\delta l)=0 \qquad \forall \delta l\in\Lzweizwei.
\end{equation}
Using the semidiscrete adjoint state $(\ovzt,\ovpt)\in\Vtau\times\Xtau$ given as the solution of the semidiscrete adjoint equation
\begin{equation}
\label{eq:semidiscreteadjointequation}
\B((\psi,\lambda),(\ovzt,\ovpt))=(\overline{\varphi}_\tau-\varphi_d,\psi)_{I\times\Omega}+(\overline{d}_\tau-d_d,\lambda)_{I\times\Omega} \qquad \forall (\psi,\lambda)\in\Vtau\times\Xtau
\end{equation}
the first order necessary optimality condition can equivalently be expressed as
\begin{equation}
\label{eq:fonlineardualsemidiscrete}
(\alpha_l \overline{l}_\tau+\overline{z}_\tau,\delta l)_{I\times\Omega}=0 \qquad \forall \delta l\in\Lzweizwei.
\end{equation}
Note, that the unique solvability of \eqref{eq:semidiscreteadjointequation} is guaranteed by Lemma \ref{cor:stabestdualtime} with right-hand sides $g_1=\overline{\varphi}_\tau-\varphi_d,g_2=\overline{d}_\tau-d_d\in\Lzweizwei$ and terminal condition $p_T=0$. Furthermore, we have $\ovlt=-\frac{1}{\alpha_l}\ovzt\in\Vtau$.

The optimization problem governed by the fully discretized PDE-ODE-System again with a continuous control reads
\begin{equation}
\label{eq:fullydiscreteoptconprob}
\min j_{\tau h}(l)=J(S_{\tau h}(l),l) \text{ subject to } l\in\Lzweizwei
\end{equation}
with $S_{\tau h}:\Lzweizwei\to \Vth\times\Xth$, $S_{\tau h}(l)=(\vth,\dth)$ being the solution operator of the fully discretized state equation \eqref{eq:fullydiscretestate}. The existence of a unique solution follows directly from the unique solvability on the semidiscrete level. Similarly, the first order necessary optimality condition
\begin{equation}
\label{eq:fonlinearprimaldiscrete}
j_{\tau h}^\prime(\overline{l}_{\tau h})(\delta l)=0 \qquad \forall \delta l\in\Lzweizwei
\end{equation}
can equivalently be expressed as
\begin{equation}
\label{eq:fonlineardualdiscrete}
(\alpha_l \overline{l}_{\tau h}+\overline{z}_{\tau h},\delta l)_{I\times\Omega}=0 \qquad \forall \delta l\in\Lzweizwei
\end{equation}
with $(\ovzth,\ovpth)\in\Vth\times\Xth$ being the solution of the fully discrete adjoint equation
\begin{equation}
\label{eq:fullydiscreteadjointequation}
\B((\psi,\lambda),(\overline{z}_{\tau h},\overline{p}_{\tau h}))=(\overline{\varphi}_{\tau h}-\varphi_d,\psi)_{I\times\Omega}+(\overline{d}_{\tau h}-d_d,\lambda)_{I\times\Omega}\qquad \forall (\psi,\lambda)\in\Vth\times\Xth.
\end{equation}
A direct consequence of \eqref{eq:fonlineardualdiscrete} is $\ovlth\in\Vth$. 
\subsection{A priori error estimates for the optimal control}
Since $\ovl\in V$ we intend to discretize the controls corresponding to the states, that is we apply a dG(0)cG(1) discretization. Thus, we choose the finite dimensional subspace $L_\sigma\subset \Lzweizwei$ as
\[L_\sigma=\{l\in L^2(\I;\Vh): l_{\vert I_m}\in\mathbb{P}_0(I_m;\Vh),1\leq m\leq M\}=\Vth.\]
Then, the fully discretized optimal control problem reads as follows
\begin{equation} \label{eq:optconprobdiscretecontrol}
\min j_{\tau h}(l)=J(S_{\tau h}(l),l) \text{ subject to } l\in L_\sigma.
\end{equation}
We want to derive a priori error estimates for the solution of this control problem. The solution will be denoted with $\ovlsigma$. Note, that we can prove $\ovlsigma=\ovlth$, a phenomenon first discussed in \cite{MH05}. Indeed, we have due to the first order optimality condition on the space-time discretized level
\[\ovlth=\frac{1}{\alpha_l}\ovzth\in\Vth=L_\sigma.\]
Thus, $\ovlth\in L_\sigma$ fulfills the necessary optimality condition on the completely discretized level for all directions $\delta l\in L_\sigma\subset \Lzweizwei$. Due to the uniqueness of the optimal control we have $\ovlsigma=\ovlth$.  Therefore, it suffices to bound the error $\Vert \ovl-\ovlth\Vert_{I\times\Omega}$ which we will do in the following. 

We start with an estimate for the error between the adjoint states on the continuous and on the space-time discretized level. We can employ the same arguments as in the previous section and therefore will only sketch the essential steps. 
\begin{lemma} \label{lem:errorestimateadjoint}
Let $(\ovz,\ovp)\in V\times X$ be the solution of the adjoint state system \eqref{eq:adjointvariationalform} on the continuous level and $(\ovzth,\ovpth)\in\Vth\times\Xth$ be the solution of the adjoint state system \eqref{eq:fullydiscreteadjointequation} on the space-time discretized level corresponding to the optimal control $\ovl\in V$ with continuous state $(\ovv,\ovd)\in V\times X$, semidiscrete state $(\ovvt,\ovdt)\in\Vtau\times\Xtau$ and space-time discrete state $(\ovvth,\ovdth)\in\Vth\times\Xth$. Then we have the following error estimate
\begin{align*}
\Vert \ovz-\ovzth\Vert_{I\times\Omega}+\Vert \ovp-\ovpth\Vert_{I\times\Omega}\leq &C\tau\{\Vert \partial_t \ovv\Vert_{I\times\Omega}+\Vert \partial_t \ovd\Vert_{I\times\Omega}+\Vert \partial_t \ovz\Vert_{I\times\Omega}+\Vert \partial_t \ovp\Vert_{I\times\Omega}\} \\
+Ch^2&\{\Vert\nabla^2\ovvt\Vert_{I\times\Omega}+\Vert\nabla^2\ovdt\Vert_{I\times\Omega}+\Vert\nabla^2
\ovzt\Vert_{I\times\Omega}+\Vert\nabla^2\ovpt\Vert_{I\times\Omega}\}.
\end{align*}
\end{lemma}
\begin{proof}
We cannot employ the arguments used in the previous section directly as we are lacking Galerkin orthogonality of the temporal and spatial errors. This is due to the discretization of the primal state variables on the right-hand side. Instead we have
\begin{equation} \label{eq:Galerkinorthotimeadjoint}
\B((\psi,\lambda),(\ezt,\ept))=(\ovv-\ovvt,\psi)_{I\times\Omega}+(\ovd-\ovdt,\lambda)_{I\times\Omega} \quad \forall (\psi,\lambda)\in \Vtau\times\Xtau
\end{equation}
for the temporal errors $\ezt=\ovz-\ovzt,\ept=\ovp-\ovpt$. Furthermore, 
\begin{equation} \label{eq:Galerkinorthospaceadjoint}
\B((\psi,\lambda),(\ezth,\epth))=(\ovvt-\ovvth,\psi)_{I\times\Omega}+(\ovdt-\ovdth,\lambda)_{I\times\Omega} \quad \forall (\psi,\lambda)\in \Vth\times\Xth
\end{equation}
holds true for the spatial errors $\ezth=\ovzt-\ovzth,\epth=\ovpt-\ovpth$.
We split the proof in several parts.
\begin{enumerate} 
\item First we derive an estimate for the temporal error. We will employ exactly the same arguments as for the temporal error for the primal variables and just use \eqref{eq:Galerkinorthotimeadjoint} whenever the temporal Galerkin orthogonality is used. \\
Thus, consider the solution $(\tilde{\varphi}_\tau,\tilde{d}_\tau)\in\Vtau\times\Xtau$ of the auxiliary variational problem
\[\B((\tilde{\varphi}_\tau,\tilde{d}_\tau),(\psi,\lambda))=(\ezt,\psi)_{I\times\Omega}+(\ept,\lambda)_{I\times\Omega} \quad \forall (\psi,\lambda)\in\Vtau\times\Xtau\]
which possesses a unique solution $(\tilde{\varphi}_\tau,\tilde{d}_\tau)\in\Vtau\times\Xtau$ according to Proposition \ref{prop: semidiscretestateexistence}.
With the usual notation  we have
\begin{align*}
\Vert\ezt\Vert^2_{I\times\Omega}+\Vert\ept\Vert^2_{I\times\Omega}&=\B((\tilde{\varphi}_\tau,\tilde{d}_\tau),(\xi^z_\tau,\xi^p_\tau))+(\ezt,\etazt)_{I\times\Omega}+(\ept,\etapt)_{I\times\Omega} \\
&=-\B((\tilde{\varphi}_\tau,\tilde{d}_\tau),(\etazt,\etapt)) + (\ovv-\ovvt,\tilde{\varphi}_\tau)_{I\times\Omega}+(\ovd-\ovdt,\tilde{d}_\tau)_{I\times\Omega} \\
&+(\ezt,\etazt)_{I\times\Omega}+(\ept,\etapt)_{I\times\Omega}.
\end{align*}
From here, we mimic the steps taken in the proof of Lemma \ref{lem:timeerrorvseta} and employ the stability estimates for primal solutions  which results in
\[\Vert\ezt\Vert_{I\times\Omega}+\Vert\ept\Vert_{I\times\Omega} \leq C\{\Vert\etazt\Vert_{I\times\Omega}+\Vert\etapt\Vert_{I\times\Omega}+\Vert \ovv-\ovvt\Vert_{I\times\Omega}+\Vert \ovd-\ovdt\Vert_{I\times\Omega}\}. \]
Thus, the already derived temporal errors estimates for the primal solutions and the usual projection error estimates yield
\[\Vert \ovz-\ovzt\Vert_{I\times\Omega}+\Vert \ovp-\ovpt\Vert_{I\times\Omega}\leq C \tau\{\Vert \partial_t \ovv\Vert_{I\times\Omega}+\Vert \partial_t \ovd\Vert_{I\times\Omega}+\Vert \partial_t \ovz\Vert_{I\times\Omega}+\Vert \partial_t \ovp\Vert_{I\times\Omega}\}.\]
\item For the error estimation of the spatial error we make use of the auxiliary problem
\begin{equation} \label{eq:auxproblemadjointspace}
\B((\tilde{\varphi}_{\tau h},\tilde{d}_{\tau h}),(\psi,\lambda))=(\ezth,\psi)_{I\times\Omega}+(\epth,\lambda)_{I\times\Omega} \quad \forall (\psi,\lambda)\in\Vth\times\Xth
\end{equation}
which possesses a unique solution $(\tilde{\varphi}_{\tau h},\tilde{d}_{\tau h})\in\Vth\times\Xth$. Then we have %with %$\etavh=\tilde{\varphi}_{\tau}-\pi_h\tilde{\varphi}_{\tau}$ and $\etadh=\tilde{d}_{\tau}-\pi_h\tilde{d}_{\tau}$
\begin{align*}
\Vert \ezth\Vert_{I\times\Omega}^2+\Vert \epth\Vert^2_{I\times\Omega}&=(\ezth,\etazh)_{I\times\Omega}+(\epth,\etaph)_{I\times\Omega}+\B((\tilde{\varphi}_{\tau h},\tilde{d}_{\tau h}),(\xizh,\xiph)) \\
&=(\ezth,\etazh)_{I\times\Omega}+(\epth,\etaph)_{I\times\Omega}-\B((\tilde{\varphi}_{\tau h},\tilde{d}_{\tau h}),(\etazh,\etaph)) \\
&+ (\ovvt-\ovvth,\tilde{\varphi}_{\tau h})_{I\times\Omega}+(\ovdt-\ovdth,\tilde{d}_{\tau h})_{I\times\Omega}.
\end{align*}
%Assuming for the moment, that we have stability estimates for the projections $\pi_h\tilde{\varphi}_{\tau}$ and $\pi_h\tilde{d}_\tau$ of the form
%\[\Vert \pi_h\tilde{\varphi}_{\tau}\Vert_I + \Vert\pi_h\tilde{d}_\tau\Vert_I\leq C\{ \Vert\ezth\Vert_I+\Vert\epth\Vert_I \}\]
We can employ the same arguments as in the proofs of Lemmata \ref{lem:projectioninbilinearformlinear} and \ref{lem:boundednesserrorspatial} and arrive at
\begin{align*}
\Vert\ezth\Vert_{I\times\Omega}^2+\Vert \epth\Vert^2_{I\times\Omega}&\leq C(\Vert \etazh\Vert_{I\times\Omega}+\Vert\ovzt-\rho_h\ovzt\Vert_{I\times\Omega}
+\Vert\etaph\Vert_{I\times\Omega}\\
&+\Vert \ovvt-\ovvth\Vert_{I\times\Omega}+\Vert \ovdt-\ovdth\Vert_{I\times\Omega})(\Vert\ezth\Vert_{I\times\Omega}+\Vert \epth\Vert_{I\times\Omega}) \\
\end{align*}
The usual projection error estimates as well as the spatial error estimates for primal solutions then lead to the desired estimate for the spatial error
\[\Vert\ezth\Vert_{I\times\Omega}+\Vert \epth\Vert_{I\times\Omega} \leq Ch^2\{\Vert\nabla^2\ovvt\Vert_{I\times\Omega}+\Vert\nabla^2\ovdt\Vert_{I\times\Omega}
+\Vert\nabla^2\ovzt\Vert_{I\times\Omega}+\Vert\nabla^2\ovpt\Vert_{I\times\Omega}\}.\]
\end{enumerate}  
\end{proof}
%Thus, to complete the proof for the error estimates for the adjoint variables we need to prove the stability estimates for the spatial projections:
%\begin{lemma}
%For the solution $(\tilde{\varphi}_\tau,\tilde{d}_\tau)\in\Vtau\times\Xtau$ of the auxiliary variational problem \eqref{eq:auxproblemadjointspace} we have the stability estimate
%\[\Vert \pi_h\tilde{\varphi}_{\tau}\Vert_I + \Vert\pi_h\tilde{d}_\tau\Vert_I\leq C\{ \Vert\ezth\Vert_I+\Vert\epth\Vert_I \}\]
%for the spatial projection onto $\Vth$ and $\Xth$, respectively.
%\end{lemma}
%\begin{proof}
%As 
%\[\Vert \pi_h\tilde{\varphi}_\tau\Vert^2_I=(\tilde{\varphi}_\tau,\pi_h\tilde{\varphi}_\tau)_I\]
%as well as
%\[\Vert \pi_h\tilde{d}_\tau\Vert^2_I=(\tilde{d}_\tau,\pi_h\tilde{d}_\tau)_I\]
%hold true due to the definition of the spatial projection $\pi_h$. Therefore, we may estimate the projections as in lemma \ref{lem:stabestpihpt} by choosing $\psi=\frac{1}{\beta}\tilde{\varphi}_\tau\in\Vth\subset\Vtau$ and $\lambda=\frac{\delta}{\beta}\pi_h\tilde{d}_\tau\in\Xth\subset\Xtau$ as test functions in \eqref{eq:auxproblemadjointspace}. The application of Cauchy-Schwarz' inequality and the stability estimates for primal solutions then yield the assertion. 
%\end{proof}
We are now in the position to give an estimate for the optimal control:
\begin{theorem}
The error between the optimal control $\ovl\in\Lzweizwei$ of the optimal control problem \eqref{eq:reducedoptimalcontrolprob} and the solution $\ovlth\in L_\sigma$ of the fully discretized optimal control problem \eqref{eq:optconprobdiscretecontrol} can be estimated as
\begin{align*}
\Vert \ovl-\ovlth\Vert_{I\times\Omega}\leq& \frac{C}{\alpha_l}\tau\{\Vert \partial_t \ovv\Vert_{I\times\Omega}+\Vert \partial_t \ovd\Vert_{I\times\Omega}+\Vert \partial_t \ovz\Vert_{I\times\Omega}+\Vert \partial_t \ovp\Vert_{I\times\Omega}\} \\
&+\frac{C}{\alpha_l}h^2\{\Vert\nabla^2\ovvt\Vert_{I\times\Omega}+
\Vert\nabla^2\ovdt\Vert_{I\times\Omega}+\Vert\nabla^2\ovzt\Vert_{I\times\Omega}
+\Vert\nabla^2\ovpt\Vert_{I\times\Omega}\}
\end{align*}
with $\ovv,\ovd,\ovvt,\ovdt$ and $\ovz,\ovp,\ovzt,\ovpt$ being the optimal states and adjoint states corresponding to $\ovl$ on the continuous and on the semidiscrete level. The constant $C>0$ is independent of the mesh size $h$ as well as the temporal discretization parameter $\tau$. 
\end{theorem}
\begin{proof}
The proof is based on ideas from \cite{MH05}. Let $\overline{z}=\overline{z}(\ovl)\in V$ be the first (continuous) adjoint state corresponding to the control $\ovl$ and let $\ovzs=\ovzth(\ovlth)\in\Vth$ be the first (discrete) adjoint state corresponding to the control $\ovlth$. Moreover, let $\ovzth=\ovzth(\ovl)\in\Vth$ be the first (discrete) state corresponding to the control $\ovl$. All other solutions are denoted by the same system. Due to the first order necessary optimality conditions \eqref{eq:fonlineardual} and \eqref{eq:fonlineardualdiscrete} we have
\begin{align*}
(\alpha_l\ovl+\ovz,\ovl-\ovlth)_{I\times\Omega}&=0 \\
(\alpha_l\ovlth+\ovzs,\ovl-\ovlth)_{I\times\Omega}&=0.
\end{align*}
Subtraction of both equations yields
\begin{align*}
\alpha_l\Vert\ovl-\ovlth\Vert^2_{I\times\Omega}&=(\ovz-\ovzs,\ovlth-\ovl)_{I\times\Omega} \\
&=(\ovz-\ovzth,\ovlth-\ovl)_{I\times\Omega}+(\ovzth-\ovzs,\ovlth-\ovl)_{I\times\Omega}.
\end{align*}
We have the following estimate for the second term due to the definition of the discrete state and adjoint equation
\begin{align*}
(\ovzth-\ovzs,\ovlth-\ovl)_{I\times\Omega}&=\B((\ovvs-\ovvth,\ovds-\ovdth),(\ovzth-\ovzs,\ovpth-\ovps)) \\
&=(\ovvth-\varphi_d,\ovvs-\ovvth)_{I\times\Omega}+(\ovdth-d_d,\ovds-\ovdth)_{I\times\Omega}\\
&-(\ovvs-\varphi_d,\ovvs-\ovvth)_{I\times\Omega}-(\ovds-d_d,\ovds-\ovdth)_{I\times\Omega} \\
&=(\ovvth-\ovvs,\ovvs-\ovvth)_{I\times\Omega}+(\ovdth-\ovds,\ovds-\ovdth)_{I\times\Omega}\leq 0.
\end{align*}
Thus, we have
\[\alpha_l\Vert\ovl-\ovlth\Vert^2_{I\times\Omega}\leq(\ovz-\ovzth,\ovlth-\ovl)_{I\times\Omega}.\]
The assertion now follows by applying Cauchy-Schwarz' inequality to the right-hand side in combination with the error estimate for adjoint states from lemma \ref{lem:errorestimateadjoint}.
\end{proof}
We have the following
\begin{corollary} \label{col:errorstateandadjointoptimalcontrol}
Let $(\ovv,\ovd)\in V\times X$ be the solution of the  state equation corresponding to the optimal control $\ovl\in\Lzweizwei$ on the continuous level and let $(\ovvs,\ovds)\in\Vth\times\Xth$ be the solution of the discrete state equation corresponding to the optimal control $\ovlth\in L_\sigma$. Then we have the error estimate
\begin{align*}
\Vert\ovv-\ovvs\Vert_{I\times\Omega} &\leq \Vert \ovv-\ovvth\Vert_{I\times\Omega}+C\Vert\ovl-\ovlth\Vert_{I\times\Omega} \\
\Vert\ovd-\ovds\Vert_{I\times\Omega} &\leq \Vert \ovd-\ovdth\Vert_{I\times\Omega}+C\Vert\ovl-\ovlth\Vert_{I\times\Omega} \\
\end{align*}
with $C>0$ being the same constant as in the stability estimate \eqref{eq:stabeststatetime}.
\end{corollary}
\section{Numerical examples}
In this section we present a numerical example and validate the proven rates of convergence for the discretization of our linear model problem for a given right-hand side $l$ as well as for an associated optimal control problem numerically. All the computations have been performed with the finite element tool box FEniCS, see \cite{LM03}. For the mere simulation of the linear model problem, consider the following example with a known solution. We set $\Omega=(0,1)^2$ and $T=1$, that is $I=[0,1]$. Furthermore, the parameters are chosen as $\alpha=1, \beta=1$ and $\delta=0.1$. For the right-hand side
\[l(t,x,y)=\sin(\pi x)\sin(\pi y)\exp(t)(\beta+2\alpha\pi^2)-\beta\frac{\beta}{\beta+\delta}\sin(\pi x)\sin(\pi y)(\exp(t)-\exp(-\frac{\beta}{\delta}t))\]
and the initial condition $d_0=0$ the solution is given by
\begin{align*}
\varphi(t,x,y)&=\sin(\pi x)\sin(\pi y)\exp(t) \\
d(t,x,y)&=\frac{\beta}{\beta+\delta}\sin(\pi x)\sin(\pi y)(\exp(t)-\exp(-\frac{\beta}{\delta}t)).
\end{align*}
We will provide error estimates in two steps. First, we keep the spatial discretization parameter $h$ fixed and refine the temporal discretization parameter $\tau$. In the second part, we fix $\tau$ and decrease the spatial discretization parameter. For simplicity, we use equidistant meshes both in space and time.  
{ \small \begin{center}
\begin{tabular}{|c|c c|c c||c c|c c|}
\hline
$h/\sqrt{2}$&\multicolumn{4}{|c||}{$2^{-8}$}&\multicolumn{4}{|c|}{$2^{-9}$} \\
\hline
$\tau$& $\Vert\varphi-\vth\Vert_{I\times\Omega}$ & EOC & $\Vert d-\dth\Vert_{I\times\Omega}$ & EOC & $\Vert\varphi-\vth\Vert_{I\times\Omega}$ & EOC & $\Vert d-\dth\Vert_{I\times\Omega}$ & EOC \\
\hline
$2^{-3}$ & 0.001617 & -    & 0.033285 & - & 0.001607 & -    & 0.033279 & - \\
$2^{-5}$ & 0.000512 & 0.82 & 0.010362 & 0.84 & 0.000502 & 0.84 & 0.010356 &0.84\\
$2^{-7}$ & 0.000148 & 0.89 & 0.002764 & 0.95 & 0.000136 & 0.94 & 0.002758 & 0.95\\
$2^{-9}$ & 5.54e-05 & 0.71 & 0.000709 & 0.98 & 3.75e-05 & 0.94 & 0.000702 &0.99\\
$2^{-11}$ & 3.75e-05 & 0.28 & 0.000186 & 0.96 & 1.38e-05 & 0.72 & 0.000178 & 0.99\\
$2^{-13}$ & 3.44e-05 & 0.06 & 5.94e-05 & 0.83 & 9.38e-06 & 0.28 & 4.66e-05 & 0.97\\
\hline
\end{tabular}
\title{Refinement of the time steps for $N=66049$(left) and $N=263169$(right) nodes}
\end{center}
\begin{center}
\begin{tabular}{|c|c c|c c||c c|c c|}
\hline
$\tau$&\multicolumn{4}{|c||}{$2^{-9}$}&\multicolumn{4}{|c|}{$2^{-12}$} \\
\hline
$h/\sqrt{2}$& $\Vert\varphi-\vth\Vert_{I\times\Omega}$ & EOC & $\Vert d-\dth\Vert_{I\times\Omega}$ & EOC & $\Vert\varphi-\vth\Vert_{I\times\Omega}$ & EOC & $\Vert d-\dth\Vert_{I\times\Omega}$ & EOC \\
\hline
$2^{-3}$ & 0.032623 & -    & 0.029018 & - & 0.032578 & -    & 0.028826 & - \\
$2^{-4}$ & 0.008523 & 1.93 & 0.007740 & 1.90 & 0.008504 & 1.93 & 0.007541 &1.93\\
$2^{-5}$ & 0.002163 & 1.97 & 0.002195 & 1.81 & 0.002150 & 1.98 & 0.001925 & 1.97\\
$2^{-6}$ & 0.000552 & 1.97 & 0.000947 & 1.21 & 0.000540 & 1.99 & 0.000507 &1.92\\
$2^{-7}$ & 0.000150 & 1.88 & 0.000741 & 0.35 & 0.000136 & 1.99 & 0.000166 & 1.61\\
$2^{-8}$ & 0.000055 & 1.44 & 0.000709 & 0.06 & 0.000035 & 1.96 & 0.000100 & 0.73\\
\hline
\end{tabular}
\title{Refinement of the spatial discretization for $M=512$(left) and $M=4096$(right) time steps}
\end{center}}

\begin{center}
\begin{figure} [h]
{\includegraphics[width=\textwidth]{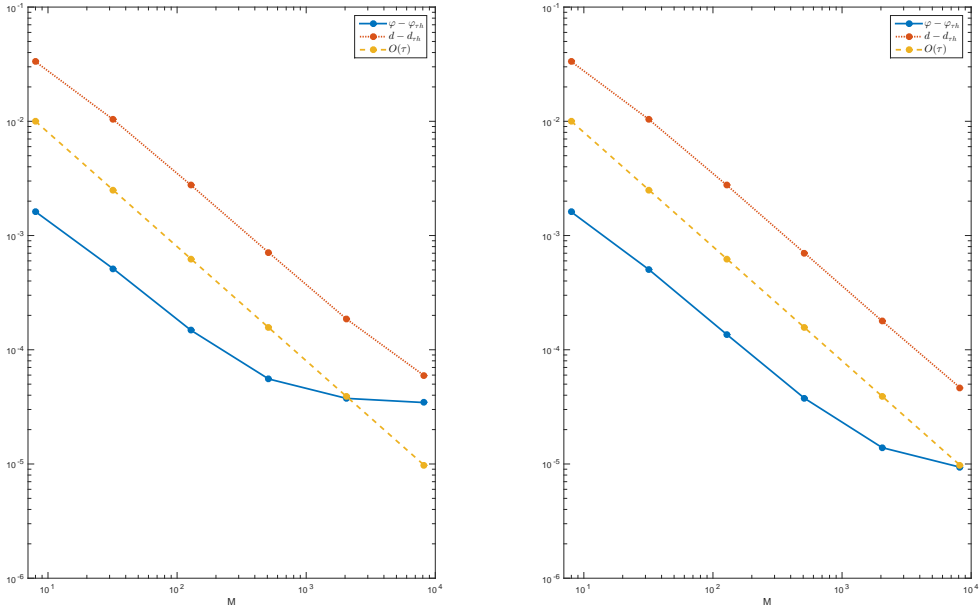} \caption{Refinement of the temporal discretization for $N=66049$(left) and $N=263169$(right) nodes }} 

\end{figure} 
\end{center}
The first table and figure 1 depict the development of the error under refinement of the temporal discretization parameter $\tau$ for two different spatial discretizations. We can see the expected linear convergence in time until the spatial discretization error becomes dominant. Moreover, we find that the error for the function $\varphi$ is smaller than the error for the function $d$ and therefore finer grids in space are needed to illustrate the stated rate of convergence in time for the function $\varphi$. The behavior of the errors switches if we fix the temporal discretization parameter and refine the spatial discretization parameter $h$. In this case we already observe the stated quadratic rate of convergence for the error of the function $\varphi$ on a more coarse time grid while we require a finer discretization in time to validate the spatial rate of convergence also for the error of the function $d$. This behavior is illustrated in the second table and figure 2, respectively. 
\begin{center}
\begin{figure} [h]
{\includegraphics[width=\textwidth]{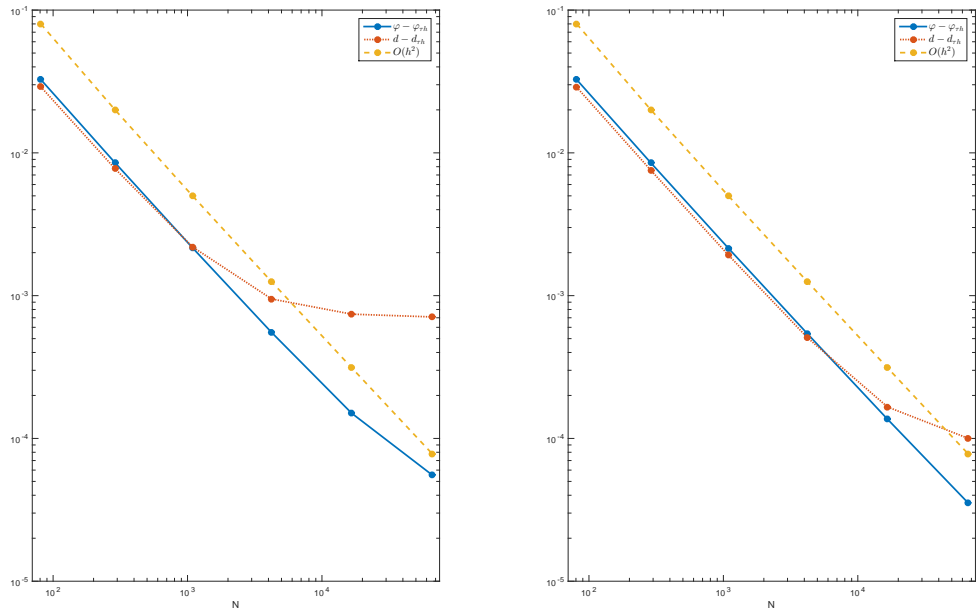} \caption{Refinement of the spatial discretization for $M=512$(left) and $M=4096$(right) time steps }} 

\end{figure} 
\end{center}
Based on this example for the simulation, we now solve an associated optimal control problem by applying a conjugate gradient method to the reduced problem. Therefore, we consider the following objective
\[J(\varphi,d,l)=\frac{1}{2}\Vert \varphi-\varphi_d\Vert^2_{I\times\Omega}+\frac{1}{2}\Vert d-d_d\Vert^2_{I\times\Omega}+\frac{1}{2}\Vert l-l_d\Vert^2_{I\times\Omega}\]
with desired states $\varphi_d=\varphi$ and $d_d=d$ given as above. Usually, $l_d$ is set to zero, but in this case we set $l_d=l$. The presence of $l_d$ alters the variational inequality in the optimality system and therefore changes the gradient of the reduced objective but it has no influence on the error estimation. Adding the function $l_d$ to the objective has the advantage, that the global solution of the optimal control problem is trivial and the minimal objective value is zero. As before we will provide the errors for the optimal control $l$ in two steps. The third table and figure 3 illustrate the proven rate of convergence for the error of the optimal control variable $l$.
{\small \begin{center}
\begin{tabular}{|c|c c|c c||c|c c|c c|}
\hline
$h/\sqrt{2}$&\multicolumn{2}{|c|}{$2^{-7}$}&\multicolumn{2}{|c||}{$2^{-8}$} & $\tau$ & \multicolumn{2}{|c|}{$2^{-9}$} & \multicolumn{2}{|c|}{$2^{-12}$} \\
\hline
$\tau$& $\Vert\ovl-\ovls\Vert_{I\times\Omega}$ & EOC & $\Vert\ovl-\ovls\Vert_{I\times\Omega}$ & EOC & $h/\sqrt{2}$& $\Vert\ovl-\ovls\Vert_{I\times\Omega}$ & EOC & $\Vert\ovl-\ovls\Vert_{I\times\Omega}$ & EOC \\
\hline
$2^{-3}$& 0.001394 & - & 0.001395 & -     &$2^{-3}$ & 0.002830 & -    & 0.002820 & - \\
$2^{-5}$&0.000439 & 0.88 & 0.000436 & 0.83  &$2^{-4}$& 0.000766 & 1.88 & 0.000757 &1.89\\
$2^{-7}$& 0.000118 & 0.94& 0.000114 & 0.96   &$2^{-5}$& 0.000203 & 1.91 & 0.000193 & 1.97\\
$2^{-9}$ & 3.49e-05& 0.88 &3.01e-05 & 0.96  &$2^{-6}$& 6.40e-05 & 1.66 & 5.01e-05 &1.95\\
$2^{-11}$ &1.73e-05 & 0.50& 1.10e-05 & 0.72  &$2^{-7}$& 3.49e-05 & 0.87 & 1.55e-05 & 1.69\\ \hline
\end{tabular}

\title{Left: Refinement of the time steps for $N=16641$ and $N=66049$ nodes, Right: Refinement of the spatial discretization for $M=512$ and $M=4096$ time steps}
\end{center}

\begin{center}
\begin{figure} [h]
{\includegraphics[width=\textwidth]{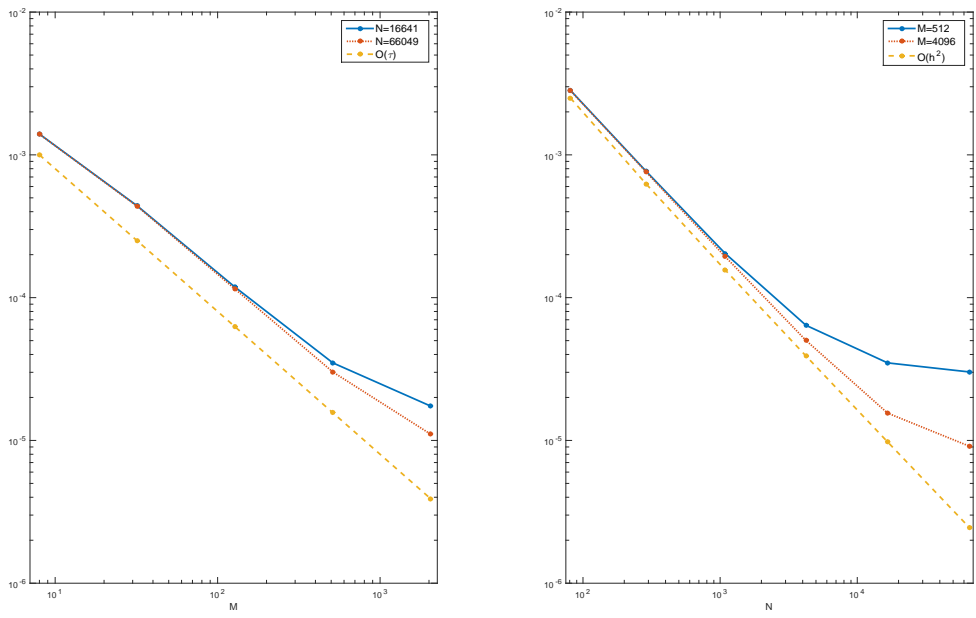} \caption{Errors $\Vert l-l_\sigma\Vert_I$ for different spatial and temporal mesh refinements }} 

\end{figure} 
\end{center}}

\bibliographystyle{abbrv}
%\bibliography{Literatur}
\bibliography{Damagepapers}

\begin{thebibliography}{10}

\bibitem{BKR17}
T.~Breiten, K.~Kunisch, and S.~Rodrigues.
\newblock {Feedback stabilization to nonstationary solutions of a class of
  reaction diffusion equations of FitzHugh-Nagumo type}.
\newblock {\em SIAM J. Control Optim.}, 55(4):2684--2713, 2017.

\bibitem{CC12}
E.~Casas and K.~Chrysafinos.
\newblock {A discontinuous Galerkin time-stepping scheme for the velocity
  tracking problem.}
\newblock {\em SIAM Journal on Numerical Analysis}, 50(5):2281--2306, 2012.

\bibitem{CC16}
E.~Casas and K.~Chrysafinos.
\newblock {Analysis of the velocity tracking control problem for the 3D
  evolutionary Navier-Stokes equations.}
\newblock {\em SIAM Journal on Control and Optimization}, 54(1):99--128, 2016.

\bibitem{CPetal2009}
K.~Chudej, H.~J. Pesch, M.~W{\"{a}}chter, G.~Sachs, and F.~{Le Bras}.
\newblock {Instationary heat-constrained trajectory optimization of a
  hypersonic space vehicle by ODE-PDE-constrained optimal control.}
\newblock {\em Variational analysis and aerospace engineering}, pages 127--144,
  2009.

\bibitem{DH08}
B.~J. Dimitrijevic and K.~Hackl.
\newblock {A method for gradient enhancement of continuum damage models}.
\newblock {\em Technische Mechanik, Ruhr-Universit{\"{a}}t Bochum},
  28(1):43--52, 2008.

\bibitem{DH11}
B.~J. Dimitrijevic and K.~Hackl.
\newblock {A regularization framework for damage-plasticity models via gradient
  enhancement of the free energy}.
\newblock {\em International Journal for Numerical Methods in Biomedical
  Engineering}, 27:1199--1210, 2011.

\bibitem{EE04}
E.~Emmrich.
\newblock {\em {Gew{\"{o}}hnliche und Operatordifferentialgleichungen}}.
\newblock Vieweg, 2004.

\bibitem{EJ91}
K.~Eriksson and C.~Johnson.
\newblock {Adaptive finite element mothods for parabolic problems I: a linear
  model problem}.
\newblock {\em SIAM Journal on Numerical Analysis}, 28(1):43--77, 1991.

\bibitem{EJ95}
K.~Eriksson and C.~Johnson.
\newblock {Adaptive finite element methods for parabolic problems II: optimal
  error estimates in $L_\infty L_2$ and $L_\infty L_\infty$}.
\newblock {\em SIAM Journal on Numerical Analysis}, 32(3):706--740, 1995.

\bibitem{EJT85}
K.~Eriksson, C.~Johnson, and V.~Thom{\'{e}}e.
\newblock {Time discretization of parabolic problems by the discontinuous
  Galerkin method}.
\newblock {\em Rairo M.M.a.N}, 19(4):611--643, 1985.

\bibitem{LE98}
L.~C. Evans.
\newblock {\em {Partial Differential Equations Vol. 19}}.
\newblock American Mathematical Society, Providence, Rhode Island, 1998.

\bibitem{KG15}
M.~Gerdts and S.-J. Kimmerle.
\newblock {Numerical optimal control of a coupled ODE-PDE model of a truck with
  a fluid basin.}
\newblock {\em Discrete Contin. Dyn. Syst. 2015, Dynamical systems,
  differential equations and applications. 10th AIMS Conference. Suppl.}, pages
  515--524, 2015.

\bibitem{GV85}
P.~Grisvard.
\newblock {\em {Elliptic Problems in Nonsmooth Domains}}.
\newblock Pitman, 1985.

\bibitem{MH05}
M.~Hinze.
\newblock {A variational discretization concept in control constrained
  optimization: The linear-quadratic case}.
\newblock {\em Computational Optimization and Applications}, 30(1):45--61,
  2005.

\bibitem{HV03}
D.~H{\"{o}}mberg and S.~Volkwein.
\newblock {Control of laser surface hardening by a reduced-order approach using
  proper orthogonal decomposition}.
\newblock {\em Mathematical and Computer Modelling}, 38(10):1003--1028, 2003.

\bibitem{KG16}
S.-J. Kimmerle and M.~Gerdts.
\newblock {Necessary optimality conditions and a semi-smooth Newton approach
  for an optimal control problem of a coupled system of Saint-Venant equations
  and ordinary differential equations}.
\newblock {\em Pure Appl. Funct. Anal.}, 1(2):231--256, 2016.

\bibitem{KGH18}
S.-J. Kimmerle, M.~Gerdts, and R.~Herzog.
\newblock {Optimal control of an elastic crane-trolley-load system - a case
  study for optimal control of coupled ODE-PDE systems}.
\newblock {\em Mathematical and Computer Modelling of Dynamical Systems},
  24(2):182--206, 2018.

\bibitem{LM03}
A.~Logg, K.~A. Mardal, and G.~N. Wells.
\newblock {\em {Automated solution of Differential Equations by the Finite
  Element Method}}.
\newblock Springer Verlag, 2012.

\bibitem{MV08}
D.~Meidner and B.~Vexler.
\newblock {A priori error estimates for space-time finite element
  discretization of parabolic optimal control Problems part I: problems without
  control constraints}.
\newblock {\em SIAM Journal on Control and Optimization}, 47(3):1150--1177,
  2008.

\bibitem{MV18}
D.~Meidner and B.~Vexler.
\newblock {Optimal Error Estimates for Fully Discrete Galerkin Approximations
  of Semilinear Parabolic Equations}.
\newblock {\em ESAIM: M2AN}, (52):2307--2325, 2019.

\bibitem{MST16}
C.~Meyer, S.~M. Schnepp, and O.~Thoma.
\newblock {Optimal control of the inhomogeneous relativistic
  Maxwell-Newton-Lorentz equations}.
\newblock {\em SIAM J. Control Optim.}, 54(5):2490--2525, 2016.

\bibitem{MS16i}
C.~Meyer and L.~Susu.
\newblock {Analysis of a viscous two-field gradient damage model, part I:
  Existence and uniqueness}.
\newblock {\em ZAA}, 2019.

\bibitem{MS16ii}
C.~Meyer and L.~Susu.
\newblock {Analysis of a viscous two-field gradient damage model, part II:
  Penalization limit}.
\newblock {\em ZAA}, 2019.

\bibitem{NV11}
I.~Neitzel and B.~Vexler.
\newblock {A priori error estimates for space-time finite element
  discretization of semilinear parabolic optimal control problems}.
\newblock {\em Numerische Mathematik}, 120(2):345--386, 2011.

\bibitem{S17}
L.~Susu.
\newblock {\em {Analysis and optimal control of a damage model with penalty}}.
\newblock PhD thesis, TU Dortmund, 2017.

\bibitem{FT09}
F.~Tr{\"{o}}ltzsch.
\newblock {\em {Optimale Steuerung partieller Differentialgleichungen}},
  volume~2.
\newblock Vieweg+Teubner, 2009.

\bibitem{WM11}
G.~Wachsmuth.
\newblock {\em {Optimal control of quasistatic plasticity}}.
\newblock PhD thesis, TU Chemnitz, 2011.

\end{thebibliography}
\end{document}